\documentclass[review, compress, 3p]{elsarticle}
\usepackage{lineno,hyperref}
\usepackage{amssymb,amsfonts}
\usepackage{amsmath,amsthm}
\usepackage{algorithm,algorithmic}
\usepackage{graphicx,subfigure}
\usepackage{float,rotating}
\usepackage{booktabs}
\usepackage{multirow,multicol,lscape}
\usepackage{epstopdf}
\usepackage{color}
\usepackage{setspace}
\usepackage{natbib}
\usepackage{bm,extarrows}
\usepackage{url,doi}

\makeatletter
\@addtoreset{equation}{section}
\makeatother

\newtheorem{theorem}{Theorem}[section]
\newtheorem{lemma}{Lemma}[section]

\newtheorem{remark}{Remark}

\journal{Elsevier}

\begin{document}

\begin{frontmatter}

\title{On the bilateral preconditioning for an L2-type all-at-once system arising
	from time-space fractional Bloch-Torrey equations}

\author[address1]{Yong-Liang Zhao}
\ead{ylzhaofde@sina.com}


\author[address2]{Xian-Ming Gu\corref{correspondingauthor}}
\ead{guxianming@live.cn}

\author[address3]{Hu Li\corref{correspondingauthor}}
\cortext[correspondingauthor]{Corresponding authors}
\ead{lihu\_0826@163.com}



\address[address1] {School of Mathematical Sciences, Sichuan Normal University, 
	Chengdu, Sichuan 610068, P.R.~China}
\address[address2] {School of Economic Mathematics, \\
	Southwestern University of Finance and Economics, Chengdu, Sichuan 611130, P.R.~China}
\address[address3] {School of Mathematics, Chengdu Normal University,
Chengdu, Sichuan 611130, P.R.~China}

\begin{abstract}
Time-space fractional Bloch-Torrey equations (TSFBTEs) are developed by some researchers to investigate 
the relationship between diffusion and fractional-order dynamics.
In this paper, we first propose a second-order implicit difference scheme for TSFBTEs by
employing the recently proposed L2-type formula [A.~A.~Alikhanov, C.~Huang, 
Appl.~Math.~Comput.~(2021) 126545].
Then, we prove the stability and the convergence of the proposed scheme.
Based on such a numerical scheme, an L2-type all-at-once system is derived.
In order to solve this system in a parallel-in-time pattern, 
a bilateral preconditioning technique is designed to accelerate the convergence of Krylov subspace solvers according to the special structure of the coefficient matrix of the system. 
We theoretically show that the condition number of the preconditioned matrix is uniformly bounded by a constant
for the time fractional order $\alpha \in (0,0.3624)$.
Numerical results are reported to show the efficiency of our method.
\end{abstract}


\end{frontmatter}


\section{Introduction}
\label{sec1}

Fractional calculus as a generalization of integer calculus fails to attract much attention until the past decades.
Due to the hereditary and memory properties of fractional derivatives, 
fractional differential equations have been successfully used in various fields 
such as electrical spectroscopy impedance \cite{lenzi2011anomalous,sun2018new}, 
earth system dynamics \cite{zhang2017review},
solute transport in porous media \cite{ghazal2018modelling} 
and image processing \cite{pu2009fractional,guo2021three}.

In physics, the diffusion model is one of important models for describing the transport process. 
The particles distributed in a normal bell-shaped pattern based on the Brownian motion are usually described by the classical diffusion model.
However, it cannot model the transport process of diffusing particles in a fractal media with locally inhomogeneous.
The reason could be that this kind of process may no longer obey the classical Fick's law.
Recently, numerous experiments show that fractional diffusion equations are more adequate than the classical one to describe anomalous diffusion \cite{metzler2000random,metzler2004restaurant,henry2006anomalous}.
Particularly, time-space fractional diffusion equations are 
generally used for modeling the anomalous diffusion, that is, 
the subdiffusion in time and the super-diffusion in space simultaneously \cite{chi2017numerical}.
In this paper, we consider the following time-space fractional Bloch-Torrey equation (TSFBTE) \cite{yu2013stability,zhu2018high}:
\begin{align}
\begin{cases}
\sideset{_0^C}{^\alpha_t}{\mathop{\mathcal{D}}} u(x,t)
= \kappa\; \frac{\partial^{\beta} u(x,t)}{\partial \left| x \right|^{\beta}} + f(x,t), 
& x \in (x_L, x_R),~t \in (0, T], \\
u(x_L,t) = u(x_R,t) = 0, & t \in [0, T],\\
u(x,0) = \phi(x), & x \in (x_L, x_R),
\end{cases}
\label{eq1.1}
\end{align}
where $0 < \alpha \leq 1$, $1 < \beta \leq 2$, the diffusion coefficient $\kappa > 0$, 
$\phi(x)$ and $f(x,t)$ are given functions.
In Eq.~\eqref{eq1.1}, $\sideset{_0^C}{^\alpha_t}{\mathop{\mathcal{D}}} u(x,t)$ and 
$\frac{\partial^{\beta} u(x,t)}{\partial \left| x \right|^{\beta}}$ 
are the Caputo and the Riesz fractional derivatives \cite{podlubny2009matrix} defined as follows:
\begin{equation*}
\sideset{_0^C}{^\alpha_t}{\mathop{\mathcal{D}}} u(x,t) =
\frac{1}{\Gamma(1 - \alpha)} \int_{0}^{t} (t - \eta)^{-\alpha}
\frac{\partial u(x,\eta)}{\partial \eta} d\eta
\end{equation*}
and 
\begin{equation*}
\frac{\partial^{\beta} u(x,t)}{\partial \left| x \right|^{\beta}} 
= -\frac{1}{2 \cos(\beta \pi/2) \Gamma(2 - \beta)} \frac{\partial^2}{\partial x^2}
\int_{- \infty}^{\infty} \left|x - \zeta \right|^{1 - \beta} u(\zeta,t) d \zeta, 
\end{equation*}
respectively. Here, $\Gamma(\cdot)$ is the Gamma function.

Eq.~\eqref{eq1.1} can be used to analyze the diffusion images of human brain tissues \cite{magin2008anomalous,yu2013numerical}.
It also provides new insights into further investigations of tissue structures and microenvironment.
Generally, it is difficult to obtain analytical solutions of fractional partial differential equations (FPDEs).
Thus, numerous numerical methods have been proposed to solve them, see \cite{li2018fast,duo2018novel,liao2019discrete,
	gu2020parallel,shen2020h2n2,zhao2021efficient,luo2021lagrange,nie2020numerical,
	chen2021explicit,zhang2021three} 
and references therein.
Yang et al.~\cite{yang2011novel} proposed two novel numerical schemes 
to solve a time-space fractional diffusion equation.
Sun et al.~\cite{sun2016some} constructed two finite difference schemes to solve Eq.~\eqref{eq1.1}.
The unique solvability, unconditional stability and convergence of their schemes are proved.
Later, Zhu and Sun \cite{zhu2018high} considered a high-order scheme for Eq.~\eqref{eq1.1}.
They proved that the convergence orders of their scheme are 3 in time and 4 in space, respectively.
Arshad et al.~\cite{arshad2017trapezoidal} proposed a second-order trapezoidal scheme 
to solve the time–space fractional diffusion equation.
In \cite{bu2015finite}, the authors proposed a numerical scheme by using a finite difference method in time
and a finite element method in space to solve the two-dimensional version of Eq.~\eqref{eq1.1}.
Dehghan and Abbaszadeh \cite{dehghan2018efficient} extended the method in \cite{bu2015finite}
to solve space/multi-time fractional Bloch-Torrey equations.
Numerical methods for solving other FPDEs can be found in 
\cite{wang2019dissipativity,zhai2019investigations,huang2020fast,wang2020dissipation,
gu2021implicit,yue2021multigrid}.

To our knowledge, we can obtain numerical solutions of FPDEs globally in time by solving
the all-at-once systems arising from FPDEs.
This advantage attracts many researchers' attentions \cite{gu2020parallel,lu2015fast,ke2015fast,huang2017fast,
	bertaccini2018limited,lin2019fast,zhao2020preconditioning,zhao2021preconditioning}.
Lu et al.~\cite{lu2015fast} proposed an approximate inversion (AI) method to solve 
the block lower triangular Toeplitz system with tri-diagonal blocks (BL3TB) from fractional sub-diffusion equations.
Ke et al.~\cite{ke2015fast} found that if the coefficient matrix in \cite{lu2015fast} is not exact BL3TB,
the AI method will be no longer available.
Thus, they \cite{ke2015fast} proposed another direct method called block divide-and-conquer (BDAC) method 
for solving the block lower triangular Toeplitz-like system with tri-diagonal block 
arising from fractional time-dependent partial differential equations.
Lin and Ng \cite{lin2019fast} developed a fast solver based on the BDAC method 
for the all-at-once system arising from the multidimensional time–space fractional
diffusion equation with variable coefficients.
Gu and Wu \cite{gu2020parallel} constructed an iterative algorithm for solving 
Volterra partial integro-differential problems with weakly singular kernel in a parallel-in-time (PinT) pattern.
Different from the above mentioned works, 
Lin et al.~\cite{lin2021parallel} developed a two-sided PinT preconditioning method 
for the all-at-once system from a non-local evolutionary equation with weakly singular kernel.
Inspired by this work, in this paper, we design a bilateral preconditioning technique
for accelerating a Krylov subspace solver to an L2-type all-at-once system arising from Eq.~\eqref{eq1.1}.

The rest of this paper is organized as follows.
In Section \ref{sec2}, we derive our L2-type all-at-once system for solving Eq.~\eqref{eq1.1}.
In Section \ref{sec3}, we propose our bilateral preconditioning technique 
and analyze the condition number of the preconditioned matrix.
Numerical results are reported in Section \ref{sec4}.
Concluding remarks are given in Section \ref{sec5}.

\section{The L2-type difference scheme and the all-at-once system}
\label{sec2}

In this section, following the idea of \cite{alikhanov2021high}, 
we propose an L2-type difference scheme for solving Eq.~\eqref{eq1.1}.
The stability and convergence of our scheme are proved.
Based on this scheme, our L2-type all-at-once system is derived.

\subsection{The L2-type difference scheme and its stability}
\label{sec2.1}

For two given positive integers $N$ and $M$, let $h = \frac{x_R - x_L}{N}$ and $\tau_t = \frac{T}{M}$.
Then, the space $[x_L, x_R]$ and the time $[0,T]$ can be discretized uniformly by 
$\Omega_h = \left\{ x_i = x_L + i h, 0 \leq i \leq N, x_0 = x_L, x_N = x_R \right\}$
and $\Omega_\tau = \left\{ t_j = j \tau_t, 0 \leq j \leq M, t_0 = 0, t_M = T \right\}$, respectively.

For approximating the Caputo fractional derivative $\sideset{_0^C}{^\alpha_t}{\mathop{\mathcal{D}}} u(x,t)$
in Eq.~\eqref{eq1.1}, we choose the following L2 formula proposed by Alikhanov and Huang \cite{alikhanov2021high}.
\begin{lemma}(\cite{alikhanov2021high}) \label{lemma2.1}
	For any $\alpha \in (0,1)$ and $y(t) \in \mathcal{C}^3[0,t_{j + 1}]~(j = 1,2,\ldots,M - 1)$.
	Then, 
	\begin{equation*}
	\left| \sideset{_0^C}{^\alpha_{t_{j + 1}}}{\mathop{\mathcal{D}}} y(t)
    - \delta_{t}^{\alpha} y(t_{j + 1}) \right| = \mathcal{O}(\tau_t^{3 - \alpha}),
	\end{equation*}
where 
\begin{equation*}
\delta_{t}^{\alpha} y(t_{j + 1})
= \frac{\tau_t^{- \alpha}}{\Gamma(2 - \alpha)} \sum_{s = 0}^{j} c^{(\alpha)}_{j - s} 
\left[ y(t_{s + 1}) - y(t_s) \right],
\end{equation*}
and for $j = 1$, 
\begin{equation*}
c^{(\alpha)}_{s}=
\begin{cases}
a_0^{(\alpha)} + b_0^{(\alpha)} + b_1^{(\alpha)}, & s = 0,\\
a_1^{(\alpha)} - b_1^{(\alpha)} - b_0^{(\alpha)}, & s = 1,
\end{cases} 
\end{equation*}

for $j = 2$,
\begin{equation*}
c^{(\alpha)}_{s}=
\begin{cases}
a_0^{(\alpha)} + b_0^{(\alpha)}, & s = 0,\\
a_1^{(\alpha)} + b_1^{(\alpha)} + b_2^{(\alpha)} - b_0^{(\alpha)}, & s = 1, \\
a_2^{(\alpha)} - b_2^{(\alpha)} - b_1^{(\alpha)}, & s = 2,
\end{cases} 
\end{equation*}

for $j \geq 3$,
\begin{equation*}
c^{(\alpha)}_{s}=
\begin{cases}
a_0^{(\alpha)} + b_0^{(\alpha)}, & s = 0,\\
a_s^{(\alpha)} + b_s^{(\alpha)} - b_{s - 1}^{(\alpha)}, & 1 \leq s \leq j - 2, \\
a_{j - 1}^{(\alpha)} + b_{j - 1}^{(\alpha)} + b_{j}^{(\alpha)} - b_{j - 2}^{(\alpha)}, & s = j - 1, \\
a_j^{(\alpha)} - b_j^{(\alpha)} - b_{j - 1}^{(\alpha)}, & s = j,
\end{cases} 
\end{equation*}
$a^{(\alpha)}_{\ell} = \left( \ell + 1 \right)^{1 - \alpha} - \ell^{1 - \alpha}~(\ell \geq 0)$
and
$b^{(\alpha)}_{\ell} = \frac{1}{2 - \alpha} \left[ \left( \ell + 1 \right)^{2 - \alpha} - \ell^{2 - \alpha}\right]
- \frac{1}{2} \left[ \left( \ell + 1 \right)^{1 - \alpha} + \ell^{1 - \alpha} \right]~(\ell \geq 0)$.
\end{lemma}

Denote 
\begin{equation*}
\Psi^{2 + \alpha}(\mathbb{R}) = \left\{ y | \int_{-\infty}^{+\infty} \left( 1 + |\omega| \right)^{2 + \alpha} 
| \hat{y}(\omega) | d \omega < \infty, ~ y \in L^1(\mathbb{R}) \right\},
\end{equation*}
where $\hat{y}(\omega) = \int_{-\infty}^{+\infty} e^{\mathrm{\mathbf{i}} \omega t} y(t) d t$ 
is the Fourier transformation of $y(t)$ and $\mathrm{\mathbf{i}} = \sqrt{-1}$.
On the other hand, the Riesz fractional derivative 
$\frac{\partial^{\beta} u(x,t)}{\partial \left| x \right|^{\beta}}$ at $x_i$ can be approximated by 
the second-order fractional centered difference method \cite{ccelik2012crank,zhang2020exponential}.
That is, for $u(x,\cdot) \in \Psi^{2 + \alpha}(\mathbb{R})$, we have
\begin{equation}
\frac{\partial^{\beta} u(x_i,t)}{\partial \left| x \right|^{\beta}} =  
-h^{-\beta} \sum_{k = -N + i}^{i} g_{k}^{\beta} u(x_{i - k}, t) + \mathcal{O}(h^2)
= \delta_{x}^{\beta} u(x_i,t) + \mathcal{O}(h^2),
\label{eq2.1}
\end{equation}
where 
\begin{equation*}
g_{k}^{\beta} = \frac{(-1)^k\; \Gamma(1 + \beta)}{\Gamma(\beta/2 - k + 1)\; \Gamma(\beta/2 + k + 1)},
\quad k \in \mathbb{Z}.
\end{equation*}

Let $u_{i}^{j}$ be the approximation of $u(x_i,t_j)$ and $f_{i}^{j} = f(x_i,t_j)$.
Combining Lemma \ref{lemma2.1} and Eq.~\eqref{eq2.1}, 
the L2-type finite difference scheme of Eq.~\eqref{eq1.1} is:
\begin{equation}
\delta_{t}^{\alpha} u_i^{j + 1} = \kappa\; \delta_{x}^{\beta} u_i^{j + 1} + f_i^{j + 1}\quad
\mathrm{for} \quad 1 \leq i \leq N - 1, ~ 1 \leq j \leq M - 1
\label{eq2.2}
\end{equation}
with the discretized boundary conditions $u_0^j = u_N^j = 0~(0 \leq j \leq M)$
and the initial value $u_i^0 = \phi(x_i)~(1 \leq i \leq N - 1)$.
Notice that the scheme \eqref{eq2.2} is not a self-starting scheme since $u_i^1$ is unknown.
Thus, we need to obtain $u_i^1$ first by other methods.
In this work, we use the following fast L1 scheme \cite{jiang2017fast,gu2021fast} to get $u_i^1$:
\begin{equation}\label{eq2.2b}
\hat{\delta}_{t}^{\alpha} \tilde{u}_i^{j} = \kappa\; \delta_{x}^{\beta} \tilde{u}_i^{j} + f_i^{j} \quad
\mathrm{for} \quad 1 \leq i \leq N - 1, ~ 1 \leq j \leq \hat{M} = \lfloor t_1/\hat{\tau} \rfloor
\end{equation} 
with $\tilde{u}_0^j = u_0^j$, $\tilde{u}_N^j = u_N^j$ and $\tilde{u}_i^0 = u_i^0$, 
where $\hat{\tau} = \tau_t^{\frac{3 - \alpha}{2 - \alpha}}$, $u_i^{1} = \tilde{u}_i^{\hat{M}}$ and 
\begin{equation*}
\hat{\delta}_{t}^{\alpha} \tilde{u}_i^{j}
= \frac{1}{\Gamma(1 - \alpha)} 
\left[ b_{j}^{(j,\alpha)} \tilde{u}_i^j - 
\sum_{k = 1}^{j - 1} \left( b_{k + 1}^{(j,\alpha)} - b_{k}^{(j,\alpha)} \right) \tilde{u}_i^k - b_{1}^{(j,\alpha)} \tilde{u}_i^0 \right].
\end{equation*}
Here 
\begin{equation*}
b_{k}^{(j,\alpha)} = 
\begin{cases}
\sum\limits_{\ell = 1}^{\hat{M}_{exp}} w_{\ell}
\int_{k - 1}^{k} e^{-\hat{\tau} s_{\ell} (j  - s)} ds, & k = 1,2,\ldots,j - 1,\\
\frac{\hat{\tau}^{-\alpha}}{1 - \alpha}, & k = j,
\end{cases}
\end{equation*}
$\hat{M}_{exp} \in \mathbb{N}^{+}$ and $w_{\ell}, s_{\ell} \geq 0$ ($\ell = 1, 2, \ldots, \hat{M}_{exp}$).

For any $\bm{v},~\bm{w} \in \mathcal{S} = \left\{ \bm{v} | \bm{v} = (v_0, v_1, \ldots, v_N), ~
v_0 = v_N = 0 \right\}$, we define an inner product and the corresponding norm:
\begin{equation*}
\left( \bm{v}, \bm{w} \right) = h \sum_{i = 1}^{N - 1} v_i w_i, \qquad
\| \bm{v} \| = \sqrt{\left( \bm{v}, \bm{v} \right)}.
\end{equation*}

Let $\bm{u}^j = [u_1^j, u_2^j, \ldots, u_{N - 1}^j]^T$ and $\bm{f}^j = [f_1^j, f_2^j, \ldots, f_{N - 1}^j]^T$. 
With these at hand, we have the following priori estimate.
\begin{theorem}\label{th2.1}
Suppose $u_i^j~(0 \leq i \leq N, 1 \leq j \leq M)$ be a solution of the scheme \eqref{eq2.2}.
Then, we have
\begin{equation*}
\tau_t \sum_{j = 1}^{M - 1} \left( \| \bm{u}^{j + 1} \|^2 + \| \Xi^{\beta} \bm{u}^{j + 1} \|^2 \right)
\leq C_1 \left( \| \bm{u}^{1} \|^2 + \| \bm{u}^{0} \|^2 + \tau_t \sum_{j = 1}^{M - 1} \| \bm{f}^{j + 1} \|^2 \right),
\end{equation*}
where $\Xi^{\beta}$ is the square root of $-\delta_{x}^{\beta}$, 
and $C_1$ is a positive constant independent of $\tau_t$ and $h$.
\end{theorem}
\begin{proof}
This proof is similar to the proof of Theorem 3.1 in \cite{alikhanov2021high}.
Thus, we omit it here.
It is worth mentioning that in this proof, the property given as follows is used:
\begin{equation*}
(-\delta_{x}^{\beta} \bm{u}^j, \bm{u}^j) = \| \Xi^{\beta} \bm{u}^j \|^2 
\geq c_{*}^{\beta} (x_R - x_L)^{-\beta}  \| \bm{u}^j \|^2,
\end{equation*}
where $c_{*}^{\beta} = 2 e^{-2} \frac{(4 - \beta) (2 - \beta) \Gamma(\beta + 1)}
{(6 + \beta) (4 + \beta) (2 + \beta) \Gamma^2(\beta/2 + 1)} (3/2 + \beta/4)^{\beta + 1}$,
see \cite{sun2016some,macias2017structure} for details.
\end{proof}
Based on Theorem \ref{th2.1}, the stability and the convergence of \eqref{eq2.2} can be proved without difficulty.

\subsection{The L2-type all-at-once system}
\label{sec2.2}

In this subsection, we derive our L2-type all-at-once system based on the scheme \eqref{eq2.2}.
Firstly, we rewrite it into the matrix form:
\begin{equation}\label{eq2.3}
\frac{h^{\beta} \tau_t^{- \alpha}}{\Gamma(2 - \alpha)} \sum_{s = 0}^{j} c^{(\alpha)}_{j - s} 
\left( \bm{u}^{s + 1} - \bm{u}^{s} \right)
+ \kappa\; G_{\beta} \bm{u}^{j + 1} = h^{\beta} \bm{f}^{j + 1},
\end{equation}
where
\begin{equation*}
G_{\beta} =
\begin{bmatrix}
g_0^{\beta} & g_{-1}^{\beta} & g_{-2}^{\beta} & \cdots & g_{3 - N}^{\beta}
& g_{2 - N}^{\beta} \\
g_1^{\beta} & g_0^{\beta} & g_{-1}^{\beta} & g_{-2}^{\beta} & \cdots & g_{3 - N}^{\beta} \\
\vdots & g_1^{\beta} & g_0^{\beta} & \ddots & \ddots & \vdots \\
\vdots & \ddots & \ddots & \ddots & \ddots & g_{-2}^{\beta} \\
g_{N - 3}^{\beta} & \ddots & \ddots & \ddots & g_{0}^{\beta} & g_{- 1}^{\beta} \\
g_{N - 2}^{\beta} & g_{N - 3}^{\beta} & \cdots & \cdots & g_{1}^{\beta} & g_{0}^{\beta}
\end{bmatrix}
\end{equation*}
is a symmetric positive definite Toeplitz matrix \cite{zhang2020exponential}.

Before deriving our all-at-once system, some notations are introduced:
$\bm{0}$ is a zero matrix with suitable size,
$I_t$ and $I_x$ are two identity matrices with orders $M - 1$ and $N - 1$, respectively.
Denote 
\begin{equation*}
\bm{u} = \left[ \left( \bm{u}^2 \right)^T,  \left( \bm{u}^3 \right)^T,\cdots, \left( \bm{u}^M \right)^T \right]^T \quad\mathrm{and}\quad
\bm{f} = \left[ \left( \bm{f}^2 \right)^T,  \left( \bm{f}^3 \right)^T,\cdots, \left( \bm{f}^M \right)^T \right]^T.
\end{equation*}
To avoid the misunderstanding, we also denote
$\tilde{c}_{0}^{(\alpha)} = a_{0}^{(\alpha)} + b_{0}^{(\alpha)} + b_{1}^{(\alpha)}$,
\begin{equation*}
\tilde{c}_{k}^{(\alpha)} = a_{k}^{(\alpha)} + b_{k}^{(\alpha)} + b_{k + 1}^{(\alpha)} - b_{k - 1}^{(\alpha)},\quad
\hat{c}_{k}^{(\alpha)} = a_{k}^{(\alpha)} - b_{k}^{(\alpha)} - b_{k - 1}^{(\alpha)},\quad
k = 1,2,\ldots,M - 1.
\end{equation*}
Then, let $A_{11} = \frac{h^{\beta} \tau_t^{- \alpha}}{\Gamma(2 - \alpha)} \tilde{c}_{0}^{(\alpha)}$, 
$A_{12} = \frac{h^{\beta} \tau_t^{- \alpha}}{\Gamma(2 - \alpha)} 
\left[ \tilde{c}_{1}^{(\alpha)} - c_{0}^{(\alpha)}, \tilde{c}_{2}^{(\alpha)} - c_{1}^{(\alpha)},
\ldots, \tilde{c}_{M - 2}^{(\alpha)} - c_{M - 3}^{(\alpha)} \right]^T$ and
\begin{equation*}
A_{22} = \frac{h^{\beta} \tau_t^{- \alpha}}{\Gamma(2 - \alpha)}
\begin{bmatrix}
c_0^{(\alpha)} & 0 & \cdots & \cdots & 0 \\
c_{1}^{(\alpha)} - c_{0}^{(\alpha)} & c_0^{(\alpha)} & \ddots & \ddots & \vdots \\
\ddots & \ddots & \ddots & \ddots & 0 \\
\ddots & \ddots & \ddots & c_0^{(\alpha)} & 0 \\
c_{M - 3}^{(\alpha)} - c_{M - 4}^{(\alpha)} & \cdots & \cdots &  c_{1}^{(\alpha)} - c_{0}^{(\alpha)} & c_0^{(\alpha)}
\end{bmatrix}.
\end{equation*}
Here $c_0^{(\alpha)} = a_0^{(\alpha)} + b_0^{(\alpha)}$ and 
$c_s^{(\alpha)} = a_s^{(\alpha)} + b_s^{(\alpha)} - b_{s - 1}^{(\alpha)}$ ($s = 1, \ldots, M - 3$).

With the help of Eq.~\eqref{eq2.3} and the above notations, 
the all-at-once system is written as:
\begin{equation}\label{eq2.4}
\mathcal{M} \bm{u} = - \bm{\eta} + h^{\beta} \bm{f},
\end{equation}
where $\mathcal{M} = A_t \otimes I_x + I_t \otimes (\kappa\; G_{\beta})$ with
\begin{equation*}
A_t = 
\begin{bmatrix}
A_{11} & \bm{0} \\
A_{12} & A_{22}
\end{bmatrix}
\end{equation*}
and 
\begin{equation*}
\bm{\eta} = \frac{h^{\beta} \tau_t^{- \alpha}}{\Gamma(2 - \alpha)}
\begin{bmatrix}
\hat{c}_{1}^{(\alpha)} \left( \bm{u}^1 - \bm{u}^0 \right) - \tilde{c}_{0}^{(\alpha)} \bm{u}^1 \\
\hat{c}_{2}^{(\alpha)} \left( \bm{u}^1 - \bm{u}^0 \right) - \tilde{c}_{1}^{(\alpha)} \bm{u}^1 \\
\vdots \\
\hat{c}_{M - 1}^{(\alpha)} \left( \bm{u}^1 - \bm{u}^0 \right) - \tilde{c}_{M - 2}^{(\alpha)} \bm{u}^1 \\
\end{bmatrix}.
\end{equation*}

Some fast algorithms are designed based on the system \eqref{eq2.4}, see \cite{gu2020parallel,zhao2021preconditioning}.
However, inspired by \cite{lin2021parallel}, in this work, we concentrate on another version of \eqref{eq2.4}.
More precisely, after doing a permutation transformation of $\bm{u}$, $\bm{\eta}$ and $\bm{f}$, we have
\begin{equation}\label{eq2.5}
\tilde{\mathcal{M}} \tilde{\bm{u}} = - \tilde{\bm{\eta}} + h^{\beta} \tilde{\bm{f}},
\end{equation}
where $\tilde{\mathcal{M}} = (\kappa\; G_{\beta}) \otimes I_t + I_x \otimes A_t$.
Notice that the order of the Kronecker product in $\mathcal{M}$ is changed.
In the next section, a bilateral preconditioning technique is proposed to fast solve Eq.~\eqref{eq2.5}.

\section{A bilateral preconditioning and the condition number of the preconditioned matrix}
\label{sec3}

If we use the Gaussian elimination based 
block forward substitution method \cite{huang2017fast} to solve Eq.~\eqref{eq2.5},
the storage requirement and the computational complexity of this method are $\mathcal{O}(N M^3 + N M^2)$ and 
$\mathcal{O}(M^2)$, respectively. 
To reduce the computational cost and accelerate solving Eq.~\eqref{eq2.5}, 
a bilateral preconditioning strategy is proposed in this section.

\subsection{The bilateral preconditioning technique and its implementation}
\label{sec3.1}

Following the idea of \cite{lin2021parallel}, our left and right preconditioners can be written as follows: 
\begin{equation} \label{eq3.1}
P_{l} = \left( \kappa\; G_{\tau} \right)^{-\frac{1}{2}} \otimes A_t 
+ \left( \kappa\; G_{\tau} \right)^{\frac{1}{2}} \otimes I_t
\end{equation}
and 
\begin{equation}\label{eq3.2}
P_{r} = \left( \kappa\; G_{\tau} \right)^{\frac{1}{2}} \otimes I_t,
\end{equation}
respectively.
Here, $G_{\tau} = G_{\beta} - H_{\beta}$ is a $\tau$-matrix \cite{bini1990new}, 
where $H_{\beta}$ is a Hankel matrix and its antidiagonals are given by
\begin{equation*}
\left[ g_{2}^{\beta}, g_{3}^{\beta}, \ldots, g_{N - 2}^{\beta}, 0, 0, 0, 
g_{N - 2}^{\beta}, \dots, g_{3}^{\beta}, g_{2}^{\beta}\right]^T.
\end{equation*}

Then, the bilateral preconditioned form of Eq.~\eqref{eq2.5} is
\begin{equation}\label{eq3.3}
\begin{cases}
P_{l}^{-1} \tilde{\mathcal{M}} P_{r}^{-1} \hat{\bm{u}} 
= P_{l}^{-1} \left( - \tilde{\bm{\eta}} + h^{\beta} \tilde{\bm{f}} \right), \\
\tilde{\bm{u}} = P_{r}^{-1} \hat{\bm{u}}. 
\end{cases}
\end{equation}

We know that an $N \times N$ Toeplitz matrix multiplies a vector 
can be done by fast Fourier transform (FFT) with $\mathcal{O}(N \log N)$ operations \cite{ng2004iterative}.
Thanks to the Toeplitz structure of $A_{22}$ and $G_{\beta}$, 
we choose a Krylov subspace method \cite{saad2003iterative} 
(e.g., the BiCGSTAB method \cite{van1992bi}) to solve Eq.~\eqref{eq3.3}.
In a Krylov subspace method, we have to compute the underlying matrix-vector product.
Thus, in the next, we aim to show how to compute the matrix-vector product 
$P_{l}^{-1} \tilde{\mathcal{M}} P_{r}^{-1} \bm{v}$ ($\bm{v}$ is a vector with suitable size) efficiently.

Obviously, the product $\bm{z} = P_{l}^{-1} \tilde{\mathcal{M}} P_{r}^{-1} \bm{v}$ can be split into 
the following three sub-steps:
\begin{equation}
\begin{cases}
\bm{v}_1 = P_{r}^{-1} \bm{v}, & \textrm{Step1}, \\
\bm{v}_2 = \tilde{\mathcal{M}} \bm{v}_1, & \textrm{Step2}, \\
\bm{z} = P_{l}^{-1} \bm{v}_2, & \textrm{Step3}.
\end{cases}
\label{eq3.4}
\end{equation}

From \cite{bini1990new}, we know that the $\tau$-matrix $G_{\tau}$ can be diagonalized as follows:
\begin{equation*}
G_{\tau} = Q_x^T D_{\tau} Q_x,
\end{equation*}
where $D_{\tau} = \mathrm{diag}(\lambda_{\tau,1}, \lambda_{\tau,2}, \ldots, \lambda_{\tau,N - 1} )$ 
is a diagonal matrix containing all eigenvalues of $G_{\tau}$, and
\begin{equation*}
Q_x = \left( \sqrt{2/N} \sin \left( \frac{i j \pi}{N} \right) \right)_{1 \leq i,j \leq N - 1}
\end{equation*}
is the sine transform matrix.
With this decomposition, Step1 in Eq.~\eqref{eq3.4} can be fast implemented in the following way:
\begin{equation*}
\bm{v}_1 = P_{r}^{-1} \bm{v}
= \left( Q_x^T \otimes I_t \right) \left[ \left( \kappa\; D_{\tau} \right)^{-\frac{1}{2}} \otimes I_t \right]
\left( Q_x \otimes I_t \right) \bm{v}.
\end{equation*}

Benefiting from properties of Kronecker product and Lemma 6 in \cite{lin2021parallel},
the storage requirement and the computational cost in Step1 are 
$\mathcal{O}(MN)$ and $\mathcal{O}(M N \log N)$, respectively.
As for Step2 in Eq.~\eqref{eq3.4}, $\tilde{\mathcal{M}} \bm{z}_1$ can be computed by using FFT 
since $A_{22}$ and $G_{\beta}$ are two Toeplitz matrices.
Thus, the computational complexity and the storage requirement in Step2 are
$\mathcal{O}(M N \log (MN))$ and $\mathcal{O}(MN)$, respectively.

Compared with the previous two steps (i.e., Step1 and Step2), the third step is a little more complicated.
Using the diagonalization of $G_{\tau}$, we can rewrite $P_{l}$ as:
\begin{equation*}
P_{l} = \left( Q_x^T \otimes I_t \right) 
\left[ \left( \kappa\; D_{\tau} \right)^{\frac{1}{2}} \otimes I_t 
+ \left( \kappa\; D_{\tau} \right)^{-\frac{1}{2}} \otimes A_t \right]
\left( Q_x \otimes I_t \right).
\end{equation*}

Denote 
\begin{equation*}
\varSigma_{i} = \left( \kappa\; \lambda_{\tau,i} \right)^{\frac{1}{2}} I_t
+ \left( \kappa\; \lambda_{\tau,i} \right)^{-\frac{1}{2}} A_t,\quad \textrm{for}~i = 1,2,\ldots,N - 1.
\end{equation*}
The product $\bm{z} = P_{l}^{-1} \bm{v}_2$ can be calculated via the following three steps:
\begin{equation}
\begin{cases}
\bm{z}_1 = \left( Q_x \otimes I_t \right) \bm{v}, & \textrm{Step-(a)}, \\
\varSigma_{n}\; \bm{z}_{2,n} = \bm{z}_{1,n}, ~1 \leq n \leq N - 1, & \textrm{Step-(b)}, \\
\bm{z} = \left(  Q_x^T \otimes I_t \right) \bm{z}_2, & \textrm{Step-(c)},
\end{cases}
\label{eq3.5}
\end{equation}
where $\bm{z}_j = \left[ \bm{z}_{j,1}^T, \bm{z}_{j,2}^T, \cdots, \bm{z}_{j,N - 1}^T \right]^T$ with $j = 1,2$.
Note that $\varSigma_{i}~(1 \leq i \leq N - 1)$ are 2-by-2 block matrices, i.e.,
\begin{equation*}
\varSigma_{i} = 
\begin{bmatrix}
\varSigma_{i,11} & \bm{0} \\
\varSigma_{i,12} & \varSigma_{i,22}
\end{bmatrix},
\end{equation*}
where 
\begin{equation*}
\varSigma_{i,11} = \left( \kappa\; \lambda_{\tau,i} \right)^{\frac{1}{2}} I_{t1} + \left( \kappa\; \lambda_{\tau,i} \right)^{-\frac{1}{2}} A_{11},\quad
\varSigma_{i,12} = \left( \kappa\; \lambda_{\tau,i} \right)^{-\frac{1}{2}} A_{12},\quad
\varSigma_{i,22} = \left( \kappa\; \lambda_{\tau,i} \right)^{\frac{1}{2}} I_{t2} + \left( \kappa\; \lambda_{\tau,i} \right)^{-\frac{1}{2}} A_{22}
\end{equation*}
and $\mathrm{blkdiag}(I_{t1}, I_{t2}) = I_t$.
Then, we have $\bm{z}_{2,n} = \varSigma_{n}^{-1}\; \bm{z}_{1,n}$ for $1 \leq n \leq N - 1$,
where
\begin{equation*}
\varSigma_{n}^{-1} = 
\begin{bmatrix}
\varSigma_{n,11}^{-1} & \bm{0} \\
- \varSigma_{n,22}^{-1} \varSigma_{n,12} \varSigma_{n,11}^{-1} & \varSigma_{n,22}^{-1}
\end{bmatrix}.
\end{equation*}

From \cite{commenges1984fast}, we know that the inverse of an invertible lower triangular Toeplitz matrix is
also an invertible lower triangular Toeplitz matrix. 
Thus, we choose the modified version of Bini's algorithm \cite{lin2004fast} to compute $\varSigma_{n,22}^{-1}$.
The computational cost and the storage requirement of \eqref{eq3.5} are 
$\mathcal{O}(M N \log (MN))$ and $\mathcal{O}(MN)$, respectively.
Consequently, the computation of \eqref{eq3.3} requires $\mathcal{O}(M N \log (MN))$ flops.
It is worth remarking that the invertibilities of $P_{l}$ and $P_{r}$ are not discussed.
We leave it to the following subsection.

\subsection{The condition number of the preconditioned matrix}
\label{sec3.2}

In this subsection, we show the nonsingularities of $P_{l}$ and $P_{r}$, 
and estimate the condition number of the preconditioned matrix $P_{l}^{-1} \tilde{\mathcal{M}} P_{r}^{-1}$.

Firstly, we prove that $A_t + A_t^T$ is positive definite.
Before doing this, the following properties are needed.
\begin{lemma}(\cite{gao2014new})\label{lemma3.1}
For $\alpha \in (0, 1)$, it holds
\begin{itemize}
	\item[(i)] {$a_k^{(\alpha)} > 0$ ($k \geq 0$) and $a_{0}^{(\alpha)} > a_1^{(\alpha)} > \ldots$;} 
	\item[(ii)] {$b_k^{(\alpha)} > 0$ ($k \geq 0$) and $b_{0}^{(\alpha)} > b_1^{(\alpha)} > \ldots$.}
\end{itemize}
\end{lemma}
\begin{lemma}\label{lemma3.2}
For $\alpha \in (0, 0.3624)$, we have
\begin{itemize}
	\item[(i)] {$c_k^{(\alpha)} > 0$ and $c_{k + 1}^{(\alpha)} - c_k^{(\alpha)} < 0$ for $k = 0,1,\ldots$;} 
	\item[(ii)] {$\tilde{c}_k^{(\alpha)} > 0$ and 
		$\tilde{c}_{k + 1}^{(\alpha)} - c_k^{(\alpha)} < 0$ for $k = 0,1,\ldots$.}
\end{itemize}
\end{lemma}
\begin{proof}
According to the proof of Lemma 2.2 in \cite{gao2014new}, we obtain that
for $\alpha \in (0, 0.3624)$,
$c_k^{(\alpha)} > 0$ ($k = 0,1,\ldots$), $c_1^{(\alpha)} - c_0^{(\alpha)} < 0$ and
$c_{k + 1}^{(\alpha)} - c_k^{(\alpha)} < 0$ ($k = 2,3,\ldots$).
From Fig.~\ref{fig1}(a), we have $c_{2}^{(\alpha)} - c_1^{(\alpha)} < 0$ for $\alpha \in (0, 0.3624)$.
Thus, the first property is true.

It is easy to check that $\tilde{c}_k^{(\alpha)} > 0$ ($k = 0,1,\ldots$) since 
$\tilde{c}_k^{(\alpha)} = c_k^{(\alpha)} + b_{k + 1}^{(\alpha)}$.
Now, we remain to show $\tilde{c}_{k + 1}^{(\alpha)} - c_k^{(\alpha)} < 0$.
Using Lemma \ref{lemma3.1}, we have
\begin{equation*}
\tilde{c}_{1}^{(\alpha)} - c_0^{(\alpha)} 
= a_1^{(\alpha)} - a_0^{(\alpha)} + b_1^{(\alpha)} - b_0^{(\alpha)}
+ b_2^{(\alpha)} - b_0^{(\alpha)} < 0.
\end{equation*}
Based on the proof of Lemma 2.3 in \cite{alikhanov2021high}, we get
$\tilde{c}_{k + 1}^{(\alpha)} - c_k^{(\alpha)} < 0$ ($k = 2,3,\ldots$).
Clearly, we can see from Fig.~\ref{fig1}(b) that for $\alpha \in (0, 0.3624)$, 
$\tilde{c}_{2}^{(\alpha)} - c_1^{(\alpha)} < 0$.
The proof is completed.
\begin{figure}[ht]
	\centering
	\subfigure[$c_{2}^{(\alpha)} - c_1^{(\alpha)}$]
	{\includegraphics[width=2.4in,height=2.23in]{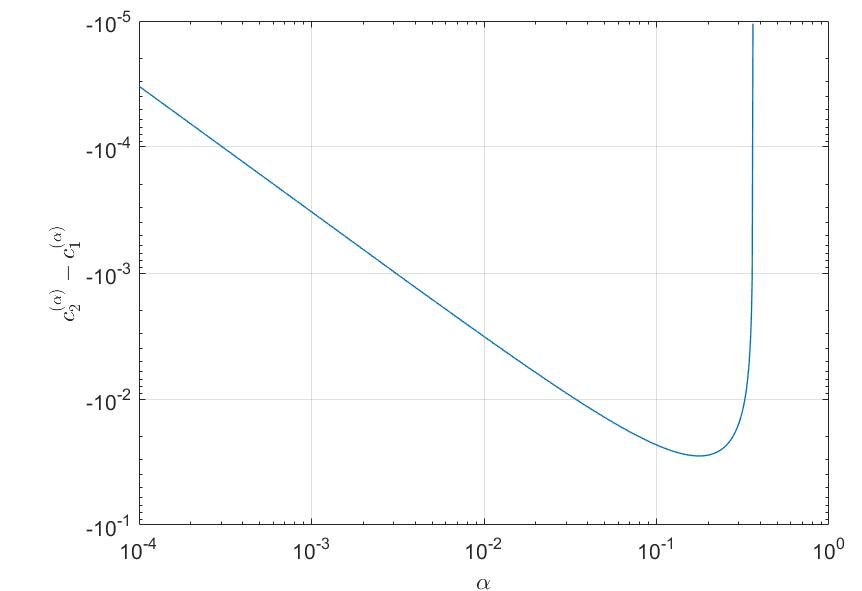}}
	\subfigure[$\tilde{c}_{2}^{(\alpha)} - c_1^{(\alpha)}$]
	{\includegraphics[width=2.4in,height=2.23in]{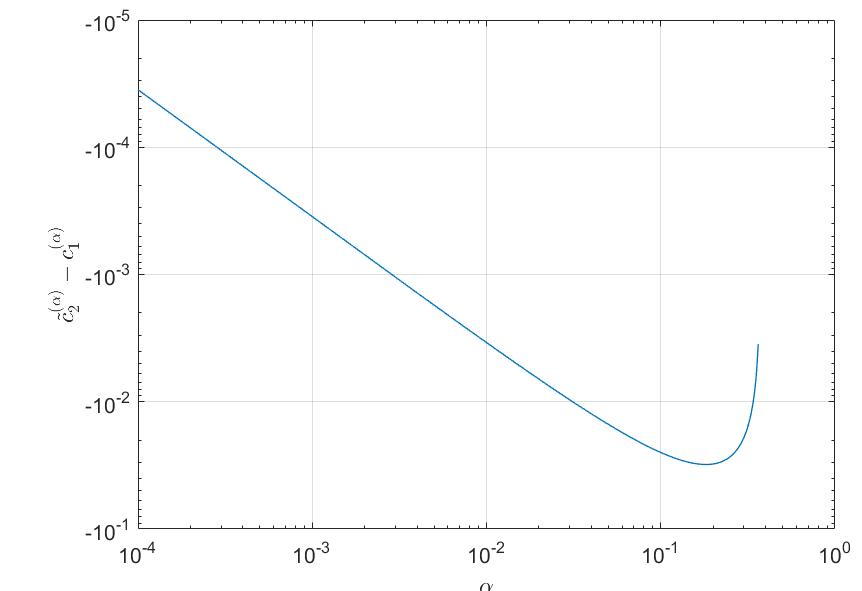}}
	\caption{The values of $c_{2}^{(\alpha)} - c_1^{(\alpha)}$ and $\tilde{c}_{2}^{(\alpha)} - c_1^{(\alpha)}$ 
		for $\alpha \in (0, 0.3624)$.}
	\label{fig1}
\end{figure}
\end{proof}

With these at hand, we get the following result.
\begin{theorem}\label{th3.1}
For any $\alpha \in (0, 0.3624)$, the matrix $A_t + A_t^T$ is positive definite.
\end{theorem}
\begin{proof}
According to Lemma \ref{lemma3.2} and Gershgorin circle theorem \cite{varga2010gervsgorin}, 
the first Gershgorin disc of the matrix $A_t + A_t^T$ is centered 
at $\frac{2 h^{\beta} \tau_t^{- \alpha}}{\Gamma(2 - \alpha)} \tilde{c}_{0}^{(\alpha)} > 0$ with radius 
\begin{equation*}
\begin{split}
R_1 & = \frac{h^{\beta} \tau_t^{- \alpha}}{\Gamma(2 - \alpha)} \sum_{k = 1}^{M - 2} 
\left| \tilde{c}_{k}^{(\alpha)} - c_{k - 1}^{(\alpha)} \right|
= \frac{h^{\beta} \tau_t^{- \alpha}}{\Gamma(2 - \alpha)} 
\left( c_{0}^{(\alpha)} + \sum_{k = 1}^{M - 3} c_{k}^{(\alpha)} - \sum_{k = 1}^{M - 2} \tilde{c}_{k}^{(\alpha)} \right) \\
& < \frac{h^{\beta} \tau_t^{- \alpha}}{\Gamma(2 - \alpha)} 
\left[ c_{0}^{(\alpha)} + \sum_{k = 1}^{M - 2} \left( c_{k}^{(\alpha)} - \tilde{c}_{k}^{(\alpha)} \right) \right]
= \frac{h^{\beta} \tau_t^{- \alpha}}{\Gamma(2 - \alpha)} 
\left( c_{0}^{(\alpha)} - \sum_{k = 2}^{M - 1} b_{k}^{(\alpha)} \right)
< \frac{2 h^{\beta} \tau_t^{- \alpha}}{\Gamma(2 - \alpha)} \tilde{c}_{0}^{(\alpha)}.
\end{split}
\end{equation*}
The $j$th ($2 \leq j \leq M - 1$) Gershgorin disc is centered 
at $\frac{2 h^{\beta} \tau_t^{- \alpha}}{\Gamma(2 - \alpha)}  c_{0}^{(\alpha)} > 0$ with radius 
\begin{equation*}
\begin{split}
R_j & = \frac{h^{\beta} \tau_t^{- \alpha}}{\Gamma(2 - \alpha)} 
\left( \left| \tilde{c}_{j - 1}^{(\alpha)} - c_{j - 2}^{(\alpha)} \right|
+ \sum_{k = 1}^{j - 2} \left| c_{k}^{(\alpha)} - c_{k - 1}^{(\alpha)} \right| 
+ \sum_{k = 1}^{M - j - 1} \left| c_{k}^{(\alpha)} - c_{k - 1}^{(\alpha)} \right| \right) \\
& = \frac{h^{\beta} \tau_t^{- \alpha}}{\Gamma(2 - \alpha)}  
\left[ c_{j - 2}^{(\alpha)} - \tilde{c}_{j - 1}^{(\alpha)} 
+ \sum_{k = 1}^{j - 2} \left( c_{k - 1}^{(\alpha)} - c_{k}^{(\alpha)} \right)
+ \sum_{k = 1}^{M - j - 1} \left( c_{k - 1}^{(\alpha)} - c_{k}^{(\alpha)} \right) \right]\\
& = \frac{h^{\beta} \tau_t^{- \alpha}}{\Gamma(2 - \alpha)}  
\left( 2 c_{0}^{(\alpha)} - \tilde{c}_{j - 1}^{(\alpha)} - c_{M - j - 1}^{(\alpha)} \right)
< \frac{2 h^{\beta} \tau_t^{- \alpha}}{\Gamma(2 - \alpha)}  c_{0}^{(\alpha)}.
\end{split}
\end{equation*}
This implies that all eigenvalues of $A_t + A_t^T$ are positive.
Thus, the matrix $A_t + A_t^T$ is positive definite.
\end{proof}

We review some properties of $g_k^{\beta}$. It will be used later.
\begin{lemma}(\cite{zhang2020exponential})\label{lemma3.3}
	Suppose $\beta \in (1,2)$, we have
	\begin{itemize}
		\item[(i)] {$g_0^{\beta} \geq 0$ and $g_k^{\beta} = g_{-k}^{\beta}$ for $|k| \geq 1$;} 
		\item[(ii)] {$g_0^{\beta} + 2 \sum\limits_{k = 1}^{N - 2} g_k^{\beta} > 0$.}
	\end{itemize}
\end{lemma}
\begin{theorem}\label{th3.2}
For any $\alpha \in (0,1)$ and $\beta \in (1,2)$, the matrices $P_{l}$ and $P_{r}$ are nonsingular.
\end{theorem}
\begin{proof}
Combining Lemma 3.2 in \cite{huang2021spectral} and Lemma \ref{lemma3.3}, 
we have $\lambda_{\tau,i} > 0$ for $1 \leq i \leq N - 1$.
Then, $P_{r}$ is nonsingular.
Noticing that $\left( \kappa\; D_{\tau} \right)^{\frac{1}{2}} \otimes I_t 
+ \left( \kappa\; D_{\tau} \right)^{-\frac{1}{2}} \otimes A_t$ is a lower triangular matrix 
with its all diagonal entries are positive.
Thus, $P_{l}$ is nonsingular and the proof is completed.
\end{proof}

The proof of Theorem \ref{th3.2} implies that the matrix $G_{\tau}$ is symmetric positive definite.
We turn to show the boundedness of the spectrum of $G_{\tau}^{-1} G_{\beta}$.
Let $h_{ij}$ and $\tilde{g}_{ij}$ be the entries of $H_{\beta}$ and $G_{\tau}$, respectively. 
Then, we get
\begin{equation*}
h_{ij} = 
\begin{cases}
g_{i + j}^{\beta}, & i + j < N - 2, \\
0, & i + j = N - 2, N - 1, N, \\ 
g_{2 N - (i + j)}^{\beta}, & \mathrm{others},
\end{cases}
\end{equation*}
and $\tilde{g}_{ij} = g_{| i - j |}^{\beta} - h_{ij}$.
Similar to \cite[Lemma 4.2]{huang2021spectral}, we know that
\begin{equation}\label{eq3.6}
p_{ii} > 0~\mathrm{and}~p_{ij} < 0 \quad \mathrm{for}~1 \leq i,j \leq N - 1, i \neq j.
\end{equation}
\begin{theorem}\label{th3.3}
The spectrum of $G_{\tau}^{-1} G_{\beta}$ is uniformly bounded,
and we have
\begin{equation*}
\frac{1}{2} < \lambda(G_{\tau}^{-1} G_{\beta}) < \frac{3}{2},
\end{equation*}
where $\lambda(G_{\tau}^{-1} G_{\beta})$ represents the eigenvalues of $G_{\tau}^{-1} G_{\beta}$.
\end{theorem}
\begin{proof}
This proof is slightly different to \cite[Theorem 4.4]{huang2021spectral}.
Thus, we omit it here.
\end{proof}

For any real symmetric matrices $D_1, D_2 \in \mathbb{R}^{n \times n}$,
we denote $D_1 \succ (\mathrm{or} \succeq) ~D_2$ if $D_1 - D_2$ is positive definite (or semi-definite).
Then, we have the following properties.
\begin{lemma}\label{lemma3.4}
\begin{itemize}
	\item[(i)] {$\bm{0} \prec \frac{1}{2} G_{\tau} \prec G_{\beta} \prec \frac{3}{2} G_{\tau}$;} 
	\item[(ii)] {$\frac{1}{2} I_x \prec G_{\tau}^{-\frac{1}{2}} G_{\beta} G_{\tau}^{-\frac{1}{2}}
		\prec \frac{3}{2} I_x$.}
\end{itemize}
\end{lemma}
\begin{proof}
By the Rayleigh quotients theorem \cite[Theorem 4.2.2]{horn2012matrix} and Theorem \ref{th3.3}, 
we have
\begin{equation*}
\frac{1}{2} < \frac{\bm{z}^T G_{\beta} \bm{z}} {\bm{z}^T G_{\tau} \bm{z}} < \frac{3}{2},
\end{equation*}
where $\bm{z} \in \mathbb{R}^{N - 1}$ is an arbitrary nonzero vector.
That is,
\begin{equation*}
\frac{1}{2} \bm{z}^T G_{\tau} \bm{z} < \bm{z}^T G_{\beta} \bm{z} < \frac{3}{2} \bm{z}^T G_{\tau} \bm{z}.
\end{equation*}
Then, we get
\begin{equation*}
\bm{0} \prec \frac{1}{2} G_{\tau} \prec G_{\beta} \prec \frac{3}{2} G_{\tau}.
\end{equation*}
On the other hand, 
\begin{equation*}
\frac{1}{2} < \frac{\bm{z}^T G_{\tau}^{-\frac{1}{2}} G_{\beta} G_{\tau}^{-\frac{1}{2}} \bm{z}}{\bm{z}^T \bm{z}}
\xlongequal{\bm{y} = G_{\beta}^{-\frac{1}{2}} \bm{z}}
\frac{\bm{y}^T G_{\beta} \bm{y}}{\bm{y}^T G_{\tau} \bm{y}} < \frac{3}{2}.
\end{equation*}
This implies (ii) and the proof is completed.
\end{proof}

For estimating the condition number of $P_{l}^{-1} \tilde{\mathcal{M}} P_{r}^{-1}$, 
another two auxiliary lemmas are needed.
\begin{lemma}(\cite{lin2021parallel})\label{lemma3.5}
Let $B_1, B_2 \in \mathbb{R}^{n \times n}$ be symmetric matrices such that 
$\bm{0} \prec B_1 \preceq B_2$. Then, $\bm{0} \prec B_2^{-1} \preceq B_1^{-1}$.
\end{lemma}
\begin{lemma}(\cite{lin2021parallel})\label{lemma3.6}
For positive numbers $\xi_k, \eta_k$ ($1 \leq k \leq n$), it holds that
\begin{equation*}
\min_{1 \leq k \leq n} \frac{\xi_k}{\eta_k} \leq \left( \sum_{k = 1}^{n} \eta_k \right)^{-1}
\left( \sum_{k = 1}^{n} \xi_k \right) \leq \max_{1 \leq k \leq n} \frac{\xi_k}{\eta_k}.
\end{equation*}	
\end{lemma}

For an arbitrary nonsingular matrix $A$, we define its condition number as
\begin{equation*}
\kappa_2(A) = \| A^{-1} \|_2 \| A \|_2.
\end{equation*}

Now, we are in position to estimate $\kappa_2(P_{l}^{-1} \tilde{\mathcal{M}} P_{r}^{-1})$.
\begin{theorem}\label{th3.4}
For any $\alpha \in (0, 0.3624)$, the condition number of $P_{l}^{-1} \tilde{\mathcal{M}} P_{r}^{-1}$ is bounded,
i.e.,
\begin{equation*}
\kappa_2(P_{l}^{-1} \tilde{\mathcal{M}} P_{r}^{-1}) < 2 \sqrt{3}.
\end{equation*}
\end{theorem}
\begin{proof}
Denote $\hat{\mathcal{M}} = (\kappa\; G_{\beta})^{\frac{1}{2}} \otimes I_t 
+ (\kappa\; G_{\beta})^{-\frac{1}{2}} \otimes A_t$.
Then, we have
\begin{equation*}
\tilde{\mathcal{M}} = \hat{\mathcal{M}} \left[ (\kappa\; G_{\beta})^{\frac{1}{2}} \otimes I_t \right]
\end{equation*} 
and 
\begin{equation}\label{eq3.7}
\left( P_{l}^{-1} \tilde{\mathcal{M}} P_{r}^{-1} \right) 
\left( P_{l}^{-1} \tilde{\mathcal{M}} P_{r}^{-1} \right)^T 
= P_{l}^{-1} \hat{\mathcal{M}} \left\{ \left[ (\kappa\; G_{\beta})^{\frac{1}{2}} 
(\kappa\; G_{\tau})^{-1} (\kappa\; G_{\beta})^{\frac{1}{2}} \right] \otimes I_t \right\}
\hat{\mathcal{M}}^T P_{l}^{-T}.
\end{equation} 

We notice that the matrix $\left[ (\kappa\; G_{\beta})^{\frac{1}{2}} 
(\kappa\; G_{\tau})^{-1} (\kappa\; G_{\beta})^{\frac{1}{2}} \right] \otimes I_t$
is similar to $\left[ (\kappa\; G_{\tau})^{-\frac{1}{2}} 
(\kappa\; G_{\beta}) (\kappa\; G_{\tau})^{-\frac{1}{2}} \right] \otimes I_t$.
Using Lemma \ref{lemma3.4}(ii), we obtain
\begin{equation*}
\frac{1}{2} I_{xt} \prec \left[ (\kappa\; G_{\beta})^{\frac{1}{2}} 
(\kappa\; G_{\tau})^{-1} (\kappa\; G_{\beta})^{\frac{1}{2}} \right] \otimes I_t
\prec \frac{3}{2} I_{xt},
\end{equation*}
where $I_{xt} = I_x \otimes I_t$.
Bring this estimate into Eq.~\eqref{eq3.7}, we get
\begin{equation}\label{eq3.8}
\frac{1}{2} P_{l}^{-1} \hat{\mathcal{M}} \hat{\mathcal{M}}^T P_{l}^{-T} 
\prec \left( P_{l}^{-1} \tilde{\mathcal{M}} P_{r}^{-1} \right) 
\left( P_{l}^{-1} \tilde{\mathcal{M}} P_{r}^{-1} \right)^T
\prec \frac{3}{2} P_{l}^{-1} \hat{\mathcal{M}} \hat{\mathcal{M}}^T P_{l}^{-T}.
\end{equation}

Now, it remains to give a bound of the matrix $P_{l}^{-1} \hat{\mathcal{M}} \hat{\mathcal{M}}^T P_{l}^{-T}$.
Let $\bm{z} \in \mathbb{R}^{N - 1}$ be an arbitrary nonzero vector.
Then, the Rayleigh quotient of $P_{l}^{-1} \hat{\mathcal{M}} \hat{\mathcal{M}}^T P_{l}^{-T}$ is
\begin{equation}\label{eq3.9}
\begin{split}
\frac{\bm{z}^T P_{l}^{-1} \hat{\mathcal{M}} \hat{\mathcal{M}}^T P_{l}^{-T} \bm{z}}{\bm{z}^T \bm{z}}
& \xlongequal{\bm{y} = P_{l}^{-T} \bm{z}}
\frac{\bm{y}^T \hat{\mathcal{M}} \hat{\mathcal{M}}^T \bm{y}}{\bm{y}^T P_{l} P_{l}^T \bm{y}} \\
& = \frac{\bm{y}^T \left[ (\kappa\; G_{\beta}) \otimes I_t
	+ I_x \otimes (A_t + A_t^T) + (\kappa\; G_{\beta})^{-1} \otimes (A_t A_t^T) \right] \bm{y}}
{\bm{y}^T \left[ (\kappa\; G_{\tau}) \otimes I_t
	+ I_x \otimes (A_t + A_t^T) + (\kappa\; G_{\tau})^{-1} \otimes (A_t A_t^T) \right] \bm{y}}.
\end{split}
\end{equation}

It is easy to check that $A_t A_t^T \succ \bm{0}$ since $A_t$ is a lower triangular matrix with
its all diagonal entries are positive. 
Moreover, from Theorem \ref{th3.1} we know that for any $\alpha \in (0, 0.3624)$, the matrix $A_t + A_t^T$
is positive definite.
Thus, we can apply Lemma \ref{lemma3.6} to estimate Eq.~\eqref{eq3.9}.

On the one hand, adopting Lemma \ref{lemma3.4}(i), we have 
\begin{equation}\label{eq3.10}
\frac{1}{2} < \frac{\bm{y}^T \left[ (\kappa\; G_{\beta}) \otimes I_t \right] \bm{y}}
{\bm{y}^T \left[ (\kappa\; G_{\tau}) \otimes I_t \right] \bm{y}}
< \frac{3}{2}.
\end{equation}

On the other hand, combining Lemma \ref{lemma3.4}(i) and Lemma \ref{lemma3.5},
we get
\begin{equation}\label{eq3.11}
\frac{2}{3} < \frac{\bm{y}^T \left[ (\kappa\; G_{\beta})^{-1} \otimes (A_t A_t^T) \right] \bm{y}}
{\bm{y}^T \left[ (\kappa\; G_{\tau})^{-1} \otimes (A_t A_t^T) \right] \bm{y}}
< 2.
\end{equation}

Applying Lemma \ref{lemma3.6} to Eq.~\eqref{eq3.9}, we have
\begin{equation*}
\frac{1}{2} = \min \left\{ \frac{1}{2}, 1, \frac{2}{3} \right\}
< \frac{\bm{z}^T P_{l}^{-1} \hat{\mathcal{M}} \hat{\mathcal{M}}^T P_{l}^{-T} \bm{z}}{\bm{z}^T \bm{z}}
< \max \left\{ \frac{3}{2}, 1, 2 \right\} = 2,
\end{equation*}
where Eqs.~\eqref{eq3.10} and \eqref{eq3.11} are used.
Then, we have 
\begin{equation*}
\frac{1}{4} I_{xt} \prec \left( P_{l}^{-1} \tilde{\mathcal{M}} P_{r}^{-1} \right) 
\left( P_{l}^{-1} \tilde{\mathcal{M}} P_{r}^{-1} \right)^T
\prec 3 I_{xt}.
\end{equation*}
This implies that
\begin{equation*}
\kappa_2(P_{l}^{-1} \tilde{\mathcal{M}} P_{r}^{-1}) < \sqrt{3 / \left( \frac{1}{4} \right)} = 2 \sqrt{3}.
\end{equation*}
\end{proof}

\begin{remark}
In Eq.~\eqref{eq2.5}, we can calculate $\bm{u}^2$ first.
Then, the rest part has block lower Toeplitz with Toeplitz blocks structure.
Following this way, we may loosen the restriction of $\alpha$ in Theorem \ref{th3.4}.
However, we have to design another preconditioner iterative method for obtaining $\bm{u}^2$.
On the other hand, our idea has another benefit: it can be extended to graded time steps 
such as \cite{zhao2021preconditioning}.
\end{remark}

\section{Numerical experiments}
\label{sec4}

In this section, we report three examples.
Example 1 shows the time and space convergence orders of the scheme \eqref{eq2.2}.
The performance of our preconditioners in Section \ref{sec3} is displayed in Example 2.
We extend the bilateral preconditioning technique to two space dimensions, and
the corresponding results are reported in Example 3.
Denote 
\begin{equation*}
Err_\infty(h, \tau) = \max_{0 \leq i \leq N, 1 \leq j \leq M} \mid e_{i}^{j} \mid, \quad
Err_2(h, \tau) = \max_{1 \leq j \leq M} \| \bm{e}^{j} \|,
\end{equation*}
where $e_{i}^{j} = u(x_i,t_j) - u_i^j$ and $\bm{e}^{j} = [e_0^j,\ldots,e_N^j]^T$.
Then, we denote
\begin{equation*}
CO_{\infty, \tau} =\log_{\tau_1/ \tau_2} \frac{Err_\infty(h, \tau_1)}{Err_\infty(h, \tau_2)}, \quad
CO_{2, \tau} =\log_{\tau_1/ \tau_2} \frac{Err_2(h, \tau_1)}{Err_2(h, \tau_2)},
\end{equation*}
\begin{equation*}
CO_{\infty, h} =\log_{h_1/ h_2} \frac{Err_\infty(h_1, \tau)}{Err_\infty(h_2, \tau)}, \quad
CO_{2, h} =\log_{h_1/ h_2} \frac{Err_2(h_1, \tau)}{Err_2(h_2, \tau)}.
\end{equation*}
Some other notations that will appear later are collected here: 
``BFSM" and ``BS" mean that MATLAB’s backslash operator
is used to solve Eqs.~\eqref{eq2.2} and \eqref{eq2.5}, respectively.
``$\mathcal{I}$" and ``$\mathcal{P}$" mean that Eq.~\eqref{eq2.5} is solved by
the BiCGSTAB method and the preconditioned BiCGSTAB (called PBiCGSTAB) method, respectively.
``Time" is the total CPU time in seconds for solving the system \eqref{eq2.5}.
``Iter1" represents the number of iterations required by 
the (P)CG (PCG means preconditioned CG) method \cite{saad2003iterative} 
for solving the fast L1 scheme \eqref{eq2.2b}.
``Iter2" is the number of iterations required by the (P)BiCGSTAB method for solving Eq.~\eqref{eq2.5}. 
The chosen Krylov subspace method is terminated if the relative residual error satisfies 
$\frac{\| \bm{r}^{(k)} \|}{\| \bm{r}^{(0)} \|} \leq 10^{-9}$
or the iteration number is more than $1000$, 
where $\bm{r}^{(k)}$ denotes the residual vector in the $k$th iteration.
The initial guess is chosen as the zero vector.
Moreover, the generalized minimum residual method \cite{saad2003iterative} is not considered in this work, 
since it requires large amounts of storage due to the orthogonalization process.

All experiments were performed on a Windows 7 (64 bit) PC-AMD PRO A10-8750B R7 CPU 3.60GHz, 16 GB of RAM using MATLAB R2017b.
\begin{remark}
It is easy to check that the coefficient matrix of \eqref{eq2.2b} is 
a time-independent symmetric positive definite Toeplitz matrix.
Thus, in our bilateral preconditioning technique, 
we choose the (P)CG method based Gohberg-Semencul formula (GSF) \cite{ng2004iterative}
for fast solving Eq.~\eqref{eq2.2b}.
However, the GSF formula cannot be applied to solve block Toeplitz systems with Toeplitz blocks.
Thus, for solving the two-dimensional problem of Eq.~\eqref{eq1.1},
we have to make a small change in the bilateral preconditioning technique.
That is, we use the (P)CG method to solve 
the two-dimensional version of \eqref{eq2.2b}.
Moreover, in this paper, we construct a $\tau$-preconditioner for \eqref{eq2.2b}, 
see \cite{bini1990new} for details.
\end{remark}

\noindent\textbf{Example 1.}
We consider Eq.~\eqref{eq1.1} with $x_L = 0$, $\kappa = x_R = T = 1$ and the source term
\begin{equation*}
\begin{split}
f(x,t) = & \left( \frac{\Gamma(4 + \alpha)}{\Gamma(4)} t^3 
+ \frac{\Gamma(3)}{\Gamma(3 - \alpha)} t^{2 - \alpha} \right) x^2 (1 - x)^2
+ \frac{\kappa \left( t^{3 + \alpha} + t^2 + 1 \right)}{2 \cos(\pi \beta/2)} \times \\
& \Big\{ \frac{\Gamma(3)}{\Gamma(3 - \beta)} \left[ x^{2 - \beta} + (1 - x)^{2 - \beta} \right] 
- \frac{2 \Gamma(4)}{\Gamma(4 - \beta)} \left[ x^{3 - \beta} + (1 - x)^{3 - \beta} \right] + \\
& \frac{\Gamma(5)}{\Gamma(5 - \beta)} \left[ x^{4 - \beta} + (1 - x)^{4 - \beta} \right] \Big\}.
\end{split}
\end{equation*}
The exact solution is $u(x,t) = \left( t^{3 + \alpha} + t^2 + 1 \right) x^2 (1 - x)^2$.
\begin{table}[ht]\tabcolsep=5.0pt
	\caption{Numerical errors and the observed time convergence orders for Example 1
		with $N = M^{(3 - \alpha)/2}$.}
	\centering
	\begin{tabular}{cccccc}
		\hline
		$(\alpha,\beta)$ & $M$ & $Err_\infty(h,\tau)$ & $CO_{\infty, \tau}$ & $Err_2(h,\tau)$ & $CO_{2, \tau}$ \\
		\hline
		(0.1, 1.5) & 10 & 3.3558E-04 & -- & 2.1958E-04 & -- \\
		           & 20 & 4.2391E-05 & 2.9848 & 2.7415E-05 & 3.0017 \\
 		           & 40 & 5.3517E-06 & 2.9857 & 3.5198E-06 & 2.9614 \\
		           & 80 & 6.7403E-07 & 2.9891 & 4.5958E-07 & 2.9371 \\
		           & 160 & 9.2695E-08 & 2.8622 & 6.0949E-08 & 2.9146 \\
		\hline
		(0.4, 1.7) & 10 & 1.0146E-03 & -- & 6.9999E-04 & -- \\
		           & 20 & 1.5215E-04 & 2.7373 & 1.0266E-04 & 2.7695 \\
		           & 40 & 2.4455E-05 & 2.6373 & 1.6239E-05 & 2.6603 \\
		           & 80 & 3.8098E-06 & 2.6823 & 2.5054E-06 & 2.6963 \\
		           & 160 & 5.9581E-07 & 2.6768 & 3.9086E-07 & 2.6803 \\
		\hline
		(0.7, 1.4) & 10 & 1.2174E-03 & -- & 8.0414E-04 & -- \\
		           & 20 & 2.4993E-04 & 2.2842 & 1.6106E-04 & 2.3198 \\
		           & 40 & 4.9078E-05 & 2.3484 & 3.1689E-05 & 2.3455 \\
		           & 80 & 9.5109E-06 & 2.3674 & 6.2572E-06 & 2.3404 \\
		           & 160 & 1.8604E-06 & 2.3540 & 1.2594E-06 & 2.3128 \\
		\hline
		(0.9, 1.9) & 10 & 4.0154E-03 & -- & 2.8992E-03 & -- \\
		           & 20 & 9.8220E-04 & 2.0315 & 7.0490E-04 & 2.0402 \\
		           & 40 & 2.3079E-04 & 2.0894 & 1.6464E-04 & 2.0981 \\
		           & 80 & 5.4304E-05 & 2.0875 & 3.8477E-05 & 2.0972 \\
		           & 160 & 1.2432E-05 & 2.1270 & 8.7530E-06 & 2.1361 \\
		\hline
	\end{tabular}
	\label{tab1}
\end{table}

Table \ref{tab1} lists the errors and the observed time convergence orders 
for different values of $\alpha$ and $\beta$.
From this table, we can see that for fixed $N = M^{(3 - \alpha)/2}$,
the observed convergence order in time is $3 - \alpha$.
Table \ref{tab2} reports the errors and the observed convergence order
in space for different values of $\alpha$ and $\beta$.
It shows that for fixed $M = 1024$, 
the errors in Table \ref{tab2} decrease steadily with increasing $N$,
and the observed convergence order in space is 2 as expected.
In a word, our method is reliable and accurate.
\begin{table}[ht]\tabcolsep=5.0pt
	\caption{Numerical errors and the observed space convergence orders for Example 1
		with $M = 1024$.}
	\centering
	\begin{tabular}{cccccc}
		\hline
		$(\alpha,\beta)$ & $N$ & $Err_\infty(h,\tau)$ & $CO_{\infty, \tau}$ & $Err_2(h,\tau)$ & $CO_{2, \tau}$ \\
		\hline
		(0.1, 1.5) & 10 & 3.1533E-03 & -- & 2.1393E-03 & -- \\
		           & 20 & 7.3035E-04 & 2.1102 & 4.8195E-04 & 2.1502 \\
		           & 40 & 1.7021E-04 & 2.1013 & 1.1044E-04 & 2.1256 \\
		           & 80 & 3.9928E-05 & 2.0918 & 2.5825E-05 & 2.0964 \\
		           & 160 & 9.4280E-06 & 2.0824 & 6.1603E-06 & 2.0677 \\
        \hline		           
		(0.4, 1.7) & 10 & 4.1944E-03 & -- & 2.9495E-03 & -- \\
		           & 20 & 9.9378E-04 & 2.0775 & 6.8541E-04 & 2.1054 \\
		           & 40	& 2.3585E-04 & 2.0751 & 1.5982E-04 & 2.1005 \\
		           & 80 & 5.6098E-05 & 2.0718 & 3.7467E-05 & 2.0928 \\
		           & 160 & 1.3377E-05 & 2.0682 & 8.8415E-06 & 2.0833 \\
		\hline
		(0.7, 1.4) & 10 & 2.4866E-03 & -- & 1.6468E-03 & -- \\
		           & 20 & 5.7380E-04 & 2.1156 & 3.7013E-04 & 2.1536 \\
		           & 40 & 1.3363E-04 & 2.1023 & 8.5825E-05 & 2.1086 \\
		           & 80 & 3.1405E-05 & 2.0892 & 2.0534E-05 & 2.0634 \\
		           & 160 & 7.4461E-06 & 2.0764 & 5.0382E-06 & 2.0270 \\
		\hline
		(0.9, 1.9) & 10 & 5.4166E-03 & -- & 3.9271E-03 & -- \\
		           & 20 & 1.3277E-03 & 2.0285 & 9.5644E-04 & 2.0377 \\
		           & 40 & 3.2529E-04 & 2.0291 & 2.3276E-04 & 2.0388 \\
		           & 80 & 7.9708E-05 & 2.0289 & 5.6655E-05 & 2.0386 \\
		           & 160 & 1.9545E-05 & 2.0279 & 1.3802E-05 & 2.0373 \\
		\hline
	\end{tabular}
	\label{tab2}
\end{table}

\noindent\textbf{Example 2.}
Consider Eq.~\eqref{eq1.1} with $\kappa = T = 1$, $x_L = -1$, $x_R = 1$ and 
the source term
\begin{equation*}
\begin{split}
f(x,t) = & \frac{\Gamma(4 + \alpha)}{\Gamma(4)} t^3 (1 + x)^2 (1 - x)^2
+ \frac{\kappa \left( t^{3 + \alpha} + 1 \right)}{2 \cos(\pi \beta/2)} 
\Big\{ \frac{4\; \Gamma(3)}{\Gamma(3 - \beta)} \left[ (1 + x)^{2 - \beta} + (1 - x)^{2 - \beta} \right] - \\
& \frac{4\; \Gamma(4)}{\Gamma(4 - \beta)} \left[ (1 + x)^{3 - \beta} + (1 - x)^{3 - \beta} \right] + 
\frac{\Gamma(5)}{\Gamma(5 - \beta)} \left[ (1 + x)^{4 - \beta} + (1 - x)^{4 - \beta} \right] \Big\}.
\end{split}
\end{equation*}
The exact solution is $u(x,t) = \left( t^{3 + \alpha} + 1 \right) (1 + x)^2 (1 - x)^2$.

Table \ref{tab3} lists the performances of methods BS, BFSM, $\mathcal{I}$ and $\mathcal{P}$.
In this table and the following tables, 
``OoM" means out of memory, ``\dag" represents that the (P)BiCGSTAB/(P)CG method does not converge 
to the desired tolerance within 1000 iterations.
Compared with the method BS, our method $\mathcal{P}$ indeed accelerates solving Eq.~\eqref{eq2.5}
and reduces the storage requirement.
Compared with the method BFSM, for large $M$ and $N$ (i.e., $M = N = 1024,2048$), 
the CPU time of the method $\mathcal{P}$ is smaller.
For the unsatisfied cases (i.e., $M = N = 128,256,512$), 
although the numbers Time of the method $\mathcal{P}$ are larger than the BFSM method,
the method $\mathcal{P}$ still has two advantages in terms of storage requirement and parallel computing.
Moreover, we notice that the numbers Iter1 and Iter2 of the method $\mathcal{P}$ are
slightly influenced by the mesh size.
It should be mentioning that from Table \ref{tab3}, for $\alpha = 0.9 (> 0.3624)$,
our method $\mathcal{P}$ still performs well.
\begin{table}[ht]\footnotesize\tabcolsep=8.0pt
	\begin{center}
		\caption{Results of various methods for $M = N$ for Example 2.}
		\centering
		\begin{tabular}{cccccccc}
			\hline
			& & \rm{BS} & \rm{BFSM} & \multicolumn{2}{c}{$\mathcal{I}$} & \multicolumn{2}{c}{$\mathcal{P}$} \\
			[-2pt] \cmidrule(lr){5-6} \cmidrule(lr){7-8} \\ [-11pt]
			$(\alpha,\beta)$ & $N$ & $\mathrm{Time}$ & $\mathrm{Time}$ & $(\mathrm{Iter1},\mathrm{Iter2})$ & $\mathrm{Time}$ & $(\mathrm{Iter1},\mathrm{Iter2})$ & $\mathrm{Time}$ \\
			\hline
			(0.1, 1.1) & 128 & 6.268 & 0.111 & (56.0, 61.0) & 1.021 & (7.0, 5.0) & 0.537 \\
			& 256 & 1160.576 & 0.524 & (80.0, 89.0) & 5.141 & (7.0, 5.0) & 1.831 \\
			& 512 & $>5$ hours & 5.097 & (114.0, 137.0) & 29.161 & (8.0, 5.0) & 5.985 \\
			& 1024 & OoM & 50.948 & (162.0, 191.0) & 261.795 & (8.0, 5.0) & 20.847 \\
			& 2048 & OoM & 537.730 & (228.0, 283.0) & 1605.622 & (8.0, 6.0) & 89.736 \\
			\\
			(0.2, 1.7) & 128 & 6.220 & 0.116 & (113.0, 233.0) & 7.692 & (6.0, 4.0) & 0.447 \\
			& 256 & 1092.552 & 0.534 & (202.0, 438.0) & 48.579 & (6.0, 5.0) & 1.873 \\
			& 512 & $>5$ hours & 5.057 & (360.0, 829.0) & 307.938 & (6.0, 5.0) & 5.786 \\
			& 1024 & OoM &	48.755 & \dag & \dag & (6.0, 5.0) & 20.128 \\
			& 2048 & OoM & 537.980 & \dag & \dag & (7.0, 6.0) & 86.305 \\
			\\
			(0.35, 1.5) & 128 & 6.210 & 0.110 & (78.0, 189.0) & 6.426 & (6.0, 5.0) & 0.539 \\
			& 256 & 1277.143 & 0.526 & (116.0, 310.0) & 34.513 & (6.0, 5.0) & 1.835 \\
			& 512 & $>5$ hours & 4.963 & (163.0, 569.0) & 212.494 & (6.0, 5.0) & 5.789 \\
			& 1024 & OoM & 51.936 & \dag & \dag & (6.0, 5.0) & 20.159 \\
			& 2048 & OoM & 582.579 & \dag & \dag & (6.0, 6.0) & 86.159 \\
			\\
			(0.9, 1.9) & 128 & 6.211 & 0.127 & \dag & \dag & (3.0, 4.0) & 0.459 \\
			& 256 & 1805.931 & 0.604 & \dag & \dag & (3.0, 4.0) & 1.534 \\
			& 512 & $>5$ hours & 5.588 & \dag & \dag & (3.0, 4.0) & 4.831 \\ 
			& 1024 & OoM & 52.250 & \dag & \dag & (3.0, 4.0) & 16.527 \\
			& 2048 & OoM & 599.093 & \dag & \dag & (3.0, 4.0) & 60.787 \\
			\hline
		\end{tabular}
		\label{tab3}
	\end{center}
\end{table}
\begin{table}[H]\footnotesize\tabcolsep=6.0pt
	\begin{center}
		\caption{The condition numbers of $\tilde{\mathcal{M}}$ and 
			$P_{l}^{-1} \tilde{\mathcal{M}} P_{r}^{-1}$ for $M = N$ for Example 2.}
		\centering
		\begin{tabular}{cccc}
			\hline
			($\alpha$, $\beta$) & $N$ & $\kappa_2(\tilde{\mathcal{M}})$ 
			& $\kappa_2(P_{l}^{-1} \tilde{\mathcal{M}} P_{r}^{-1})$ \\
			\hline
			(0.1, 1.1) & 16 & 9.86 & 1.23 \\ 
			& 32 & 20.63 & 1.30 \\
			& 64 & 43.64 & 1.36 \\
			& 128 & 92.89 & 1.42 \\
			\\
			(0.2, 1.7) & 16 & 38.04 & 1.12 \\
			& 32 & 123.25 & 1.15 \\
			& 64 & 400.27 & 1.18 \\
			& 128 & 1300.85 & 1.21 \\
			\\
			(0.35, 1.5) & 16 & 25.02 & 1.17 \\
			& 32 & 68.98 & 1.22 \\
			& 64 & 192.69 & 1.27 \\
			& 128 & 541.93 & 1.31 \\
			\\
			(0.9, 1.9) & 16 & 70.45 & 1.04 \\
			& 32 & 243.78 & 1.06 \\
			& 64 & 870.27 & 1.07 \\
			& 128 & 3171.08 & 1.08 \\
			\hline
		\end{tabular}
		\label{tab4}
	\end{center}
\end{table}

Table \ref{tab4} lists the condition numbers of $\tilde{\mathcal{M}}$ and 
$P_{l}^{-1} \tilde{\mathcal{M}} P_{r}^{-1}$ for different values of $\alpha$ and $\beta$.
From this table, we see that $\kappa_2(P_{l}^{-1} \tilde{\mathcal{M}} P_{r}^{-1}) < 2 \sqrt{3}$
for $\alpha < 0.3624$.
This is in good agreement with our theoretical analysis in Section \ref{sec3.2}.
Although we fail to give the bound of $\kappa_2(P_{l}^{-1} \tilde{\mathcal{M}} P_{r}^{-1})$ for $\alpha > 0.3624$ theoretically,
Table \ref{tab4} shows that for $\alpha = 0.9 (> 0.3624)$, 
the condition number of $P_{l}^{-1} \tilde{\mathcal{M}} P_{r}^{-1}$ is still less than $2 \sqrt{3}$.
Fig.~\ref{fig2} shows the spectrum of $\tilde{\mathcal{M}}$, $P_{l}^{-1} \tilde{\mathcal{M}}$,
$\tilde{\mathcal{M}} P_{r}^{-1}$ and $P_{l}^{-1} \tilde{\mathcal{M}} P_{r}^{-1}$.
From this figure, the eigenvalues of $P_{l}^{-1} \tilde{\mathcal{M}} P_{r}^{-1}$
are more clustered around $1$ than $P_{l}^{-1} \tilde{\mathcal{M}}$ and $\tilde{\mathcal{M}} P_{r}^{-1}$.
In a word, Table \ref{tab3} and Fig.~\ref{fig2} indicate that
our preconditioning technique is reliable and efficient for solving Eq.~\eqref{eq1.1}.
\begin{figure}[ht]
	\centering
	\subfigure[Eigenvalues of $\tilde{\mathcal{M}}$]
	{\includegraphics[width=2.4in,height=2.03in]{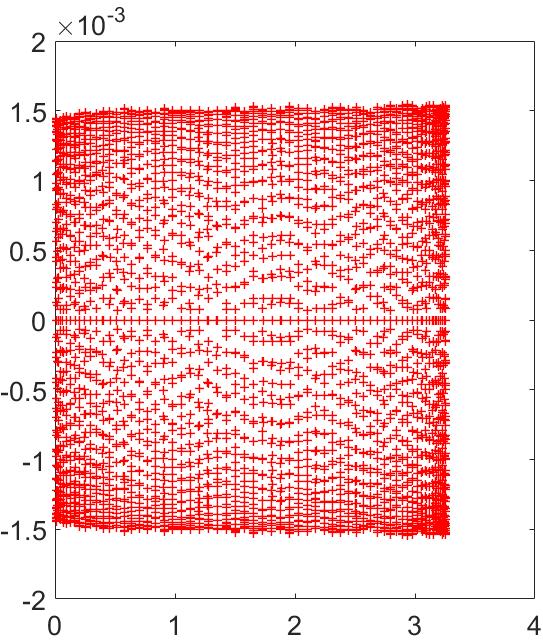}}\hspace{4mm}
	\subfigure[Eigenvalues of $P_{l}^{-1} \tilde{\mathcal{M}}$]
	{\includegraphics[width=2.4in,height=2.03in]{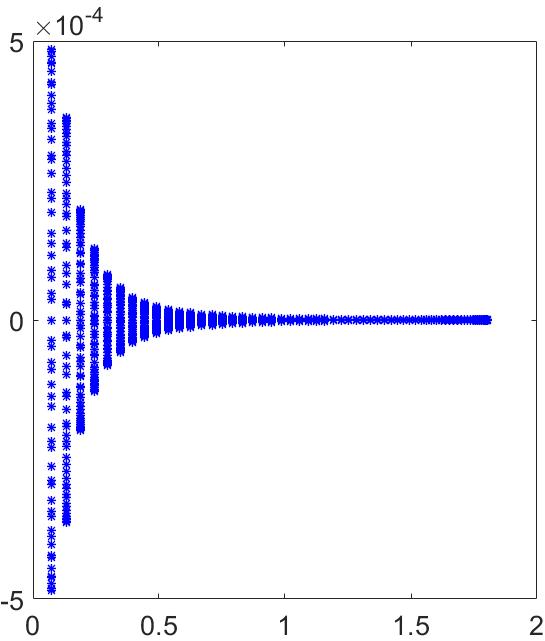}} \\
	\subfigure[Eigenvalues of $\tilde{\mathcal{M}} P_{r}^{-1}$]
	{\includegraphics[width=2.5in,height=2.03in]{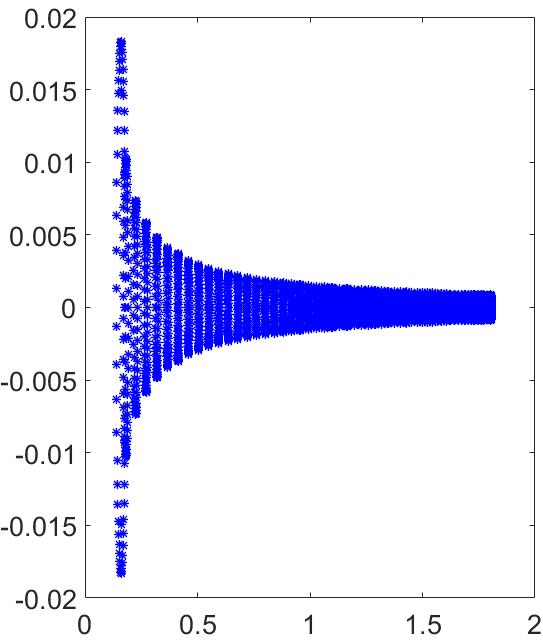}}\hspace{4mm}
	\subfigure[Eigenvalues of $P_{l}^{-1} \tilde{\mathcal{M}} P_{r}^{-1}$]
	{\includegraphics[width=2.5in,height=2.1in]{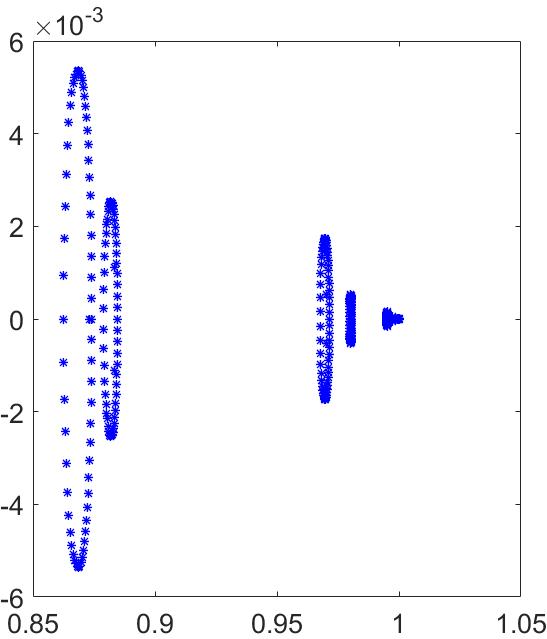}}
	\caption{Spectrum of $\tilde{\mathcal{M}}$, $P_{l}^{-1} \tilde{\mathcal{M}}$,
		$\tilde{\mathcal{M}} P_{r}^{-1}$ and $P_{l}^{-1} \tilde{\mathcal{M}} P_{r}^{-1}$
		for $(\alpha, \beta) = (0.2,1.7)$ and $M = N = 64$ for Example 2.}
	\label{fig2}
\end{figure}

\noindent\textbf{Example 3.}
In this example, we extend our bilateral preconditioning technique 
to solve the following two-dimensional problem of Eq.~\eqref{eq1.1}:
\begin{align*}
\begin{cases}
\sideset{_0^C}{^\alpha_t}{\mathop{\mathcal{D}}} u(x,y,t)
= 0.5\, \left( \frac{\partial^{\beta} u(x,y,t)}{\partial \left| x \right|^{\beta}} + 
\frac{\partial^{\beta} u(x,y,t)}{\partial \left| y \right|^{\beta}} \right)
+ f(x,y,t), 
& (x,y) \in \Omega,~t \in (0, 1], \\
u(x,y,t) = 0, & (x,y) \in \partial \Omega, ~t \in [0, 1],\\
u(x,y,0) = \phi(x,y), & (x,y) \in \Omega,
\end{cases}
\end{align*}
where $\Omega = (0,1)^2$, $\partial \Omega$ is the boundary of $\Omega$,
$\phi(x,y) = 200\, x^2 (1 - x)^2 y^2 (1 - y)^2$ and 
\allowdisplaybreaks[4]
\begin{align*}
f(x,y,t) = & \frac{200\, \Gamma(3 + \alpha + \beta)}{\Gamma(3 + \beta)} t^{2 + \beta} x^2 (1 - x)^2 y^2 (1 - y)^2
+ \frac{50 \left( t^{2 + \alpha + \beta} + 1 \right)}{\cos(\pi \beta/2)} \times \\
& \Big\{ \frac{\Gamma(3)}{\Gamma(3 - \beta)} \left[ x^{2 - \beta} + (1 - x)^{2 - \beta} \right] - 
\frac{2\; \Gamma(4)}{\Gamma(4 - \beta)} \left[ x^{3 - \beta} + (1 - x)^{3 - \beta} \right] + \\
& \frac{\Gamma(5)}{\Gamma(5 - \beta)} \left[ x^{4 - \beta} + (1 - x)^{4 - \beta} \right] \Big\} y^2 (1 - y)^2 +
\frac{50 \left( t^{2 + \alpha + \beta} + 1 \right)}{\cos(\pi \beta/2)} \times \\
& \Big\{ \frac{\Gamma(3)}{\Gamma(3 - \beta)} \left[ y^{2 - \beta} + (1 - y)^{2 - \beta} \right] -
\frac{2\; \Gamma(4)}{\Gamma(4 - \beta)} \left[ y^{3 - \beta} + (1 - y)^{3 - \beta} \right] + \\
& \frac{\Gamma(5)}{\Gamma(5 - \beta)} \left[ y^{4 - \beta} + (1 - y)^{4 - \beta} \right] \Big\} x^2 (1 - x)^2.
\end{align*}
The exact solution is $u(x,y,t) = 200 \left( t^{2 + \alpha + \beta} + 1 \right) x^2 (1 - x)^2 y^2 (1 - y)^2$.

Let $N_x$ and $N_y$ be the number of grid points in $x$- and $y$-direction, respectively.
In this example, we fix $N_x = N_y = N$. 
Table \ref{tab5} lists the results of various methods (i.e., BS, BFSM, $\mathcal{I}$ and $\mathcal{P}$) 
for different values of $\alpha$ and $\beta$.
This table indicates that the proposed method greatly reduces the storage requirement and CPU time.
For large $N$, that is $N = 32,64,128,256$, the numbers Time of the method $\mathcal{P}$ are smaller 
than the method BFSM.
It is worth mentioning that the number of iterations (i.e., $\mathrm{Iter1}$ and $\mathrm{Iter2}$) 
required by the method $\mathcal{P}$ is slightly dependent on the mesh size.
From Table \ref{tab6}, we can see that the condition number of $P_{l}^{-1} \tilde{\mathcal{M}} P_{r}^{-1}$
is less than $2 \sqrt{3}$, even for $\alpha = 0.9(> 0.3624)$.
The spectrum of $\tilde{\mathcal{M}}$, $P_{l}^{-1} \tilde{\mathcal{M}}$,
$\tilde{\mathcal{M}} P_{r}^{-1}$ and $P_{l}^{-1} \tilde{\mathcal{M}} P_{r}^{-1}$
for $(\alpha, \beta) = (0.9,1.9)$ are drawn in Fig.~\ref{fig3}.
From this figure, the eigenvalues of $P_{l}^{-1} \tilde{\mathcal{M}}$,
$\tilde{\mathcal{M}} P_{r}^{-1}$ and $P_{l}^{-1} \tilde{\mathcal{M}} P_{r}^{-1}$
are all clustered around $1$. 
The eigenvalues of $P_{l}^{-1} \tilde{\mathcal{M}} P_{r}^{-1}$ 
are the most clustered one among them.
Moreover, Fig.~\ref{fig4} compares the exact solution and numerical solution 
for $(\alpha, \beta) = (0.35,1.5)$.
It indicates that our numerical method in Section \ref{sec2.1} is accurate.
\begin{table}[H]\footnotesize\tabcolsep=8.0pt
	\begin{center}
		\caption{Results of various methods for $M = N$ for Example 3.}
		\centering
		\begin{tabular}{cccccccc}
			\hline
			& & \rm{BS} & \rm{BFSM} & \multicolumn{2}{c}{$\mathcal{I}$} & \multicolumn{2}{c}{$\mathcal{P}$} \\
			[-2pt] \cmidrule(lr){5-6} \cmidrule(lr){7-8} \\ [-11pt]
			$(\alpha,\beta)$ & $N$ & $\mathrm{Time}$ & $\mathrm{Time}$ & $(\mathrm{Iter1},\mathrm{Iter2})$ & $\mathrm{Time}$ & $(\mathrm{Iter1},\mathrm{Iter2})$ & $\mathrm{Time}$ \\
			\hline
			(0.1, 1.1) & 16 & 15.838 & 0.075 & (18.0, 19.0) & 0.584 & (6.0, 4.0) & 0.202 \\
			& 32 & $>5$ hours & 1.789 & (27.0, 33.0) & 2.333 & (7.0, 5.0) & 1.014 \\
			& 64 & $>5$ hours & 114.623 & (42.0, 51.0) & 18.085 & (7.0, 5.0) & 5.264 \\
			& 128 & $>5$ hours & $\approx 3.7$ hours & (62.0, 69.0) & 205.295 & (8.0, 5.0) & 38.762 \\
			& 256 & OoM & $>5$ hours & (91.0, 92.0) & 1913.149 & (9.0, 5.0) & 358.067 \\
			\\
			(0.2, 1.7) & 16 & 16.083 & 0.084 & (25.0, 36.0) & 0.524 & (5.0, 4.0) & 0.182 \\
			& 32 & $>5$ hours & 1.809 & (47.0, 72.0) & 5.101 & (6.0, 4.0) & 0.798 \\
			& 64 & $>5$ hours & 111.289 & (85.0, 125.0) & 45.596 & (6.0, 4.0) & 4.339 \\
			& 128 & $>5$ hours & $\approx 3.7$ hours & (154.0, 293.0) & 860.445 & (7.0, 4.0) & 31.485 \\
			& 256 & OoM & $>5$ hours & (278.0, 544.0) & $\approx 3.1$ hours & (7.0, 5.0) & 354.081 \\
			\\		
			(0.35, 1.5) & 16 & 15.783 & 0.060 & (22.0, 40.0) & 0.574 & (5.0, 4.0) & 0.187 \\
			& 32 & $>5$ hours & 1.811 & (37.0, 64.0) & 4.718 & (6.0, 5.0) & 0.997 \\
			& 64 & $>5$ hours & 110.938 & (60.0, 108.0) & 39.431 & (7.0, 5.0) & 5.282 \\
			& 128 & $>5$ hours & $\approx 3.7$ hours & (97.0, 186.0) & 549.375 & (7.0, 5.0) & 39.452 \\
			& 256 & OoM & $>5$ hours & (154.0, 333.0) & 6901.452 & (8.0, 5.0) & 359.428 \\
			\\		
			(0.9, 1.9) & 16 & 15.910 & 0.075 & (18.0, 72.0) & 1.073 & (4.0, 3.0) & 0.173 \\
			& 32 & $>5$ hours & 1.910 & (22.0, 202.0) & 13.323 & (4.0, 3.0) & 0.740 \\
			& 64 & $>5$ hours & 112.491 & (22.0, 644.0) & 229.510 & (4.0, 4.0) & 4.962 \\
			& 128 & $>5$ hours & $\approx 3.7$ hours & \dag & \dag & (3.0, 4.0) & 36.029 \\
			& 256 & OoM & $>5$ hours & \dag & \dag & (3.0, 4.0) & 323.964 \\
			\hline
		\end{tabular}
		\label{tab5}
	\end{center}
\end{table}
\begin{table}[ht]\footnotesize\tabcolsep=6.0pt
	\begin{center}
		\caption{The condition numbers of $\tilde{\mathcal{M}}$ and 
			$P_{l}^{-1} \tilde{\mathcal{M}} P_{r}^{-1}$ for $M = N$ for Example 3.}
		\centering
		\begin{tabular}{cccc}
			\hline
			($\alpha$, $\beta$) & $N$ & $\kappa_2(\tilde{\mathcal{M}})$ 
			& $\kappa_2(P_{l}^{-1} \tilde{\mathcal{M}} P_{r}^{-1})$ \\
			\hline
			(0.1, 1.1) & 8 & 5.83 & 1.20 \\
			& 16 & 12.50 & 1.28 \\
			& 32 & 26.71 & 1.36 \\ 
			\\
			(0.2, 1.7) & 8 & 26.71 & 1.36 \\
			& 16 & 49.93 & 1.14 \\
			& 32 & 163.39 & 1.18 \\
			\\
			(0.35, 1.5) & 8 & 11.40 & 1.15 \\
			& 16 & 32.43 & 1.21 \\
			& 32 & 91.73 & 1.27 \\
			\\
			(0.9, 1.9) & 8 & 22.78 & 1.04 \\
			& 16 & 22.78 & 1.04 \\
			& 32 & 306.82 & 1.07 \\
			\hline
		\end{tabular}
		\label{tab6}
	\end{center}
\end{table}

\begin{figure}[H]
	\centering
	\subfigure[Eigenvalues of $\tilde{\mathcal{M}}$]
	{\includegraphics[width=2.4in,height=2.03in]{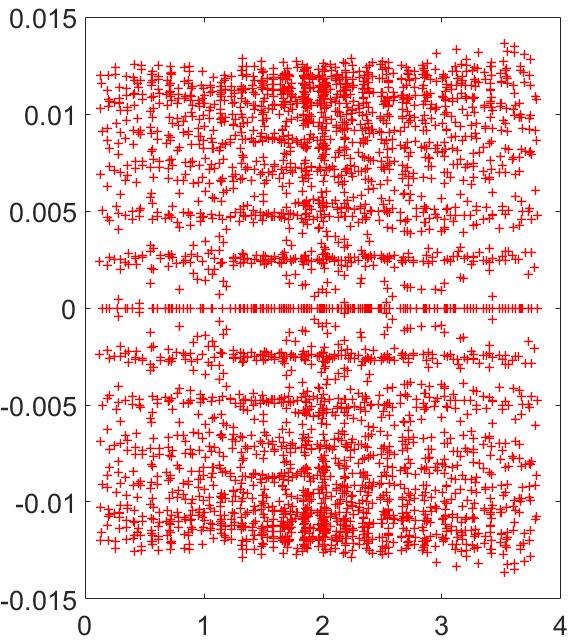}}\hspace{4mm}
	\subfigure[Eigenvalues of $P_{l}^{-1} \tilde{\mathcal{M}}$]
	{\includegraphics[width=2.4in,height=2.1in]{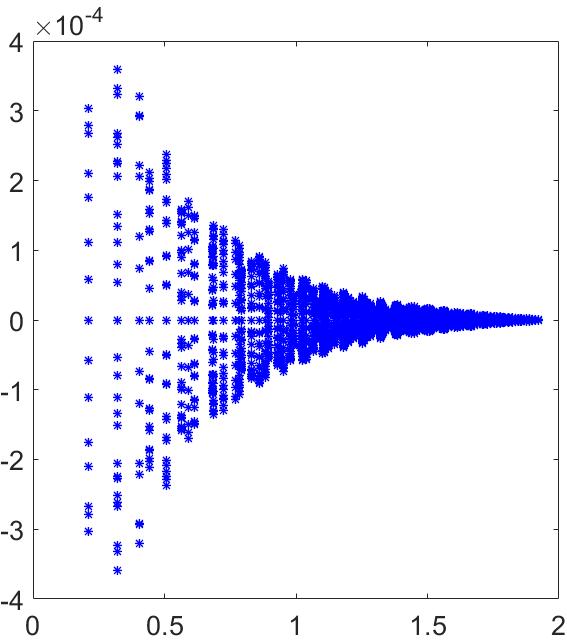}} \\
	\subfigure[Eigenvalues of $\tilde{\mathcal{M}} P_{r}^{-1}$]
	{\includegraphics[width=2.5in,height=2.03in]{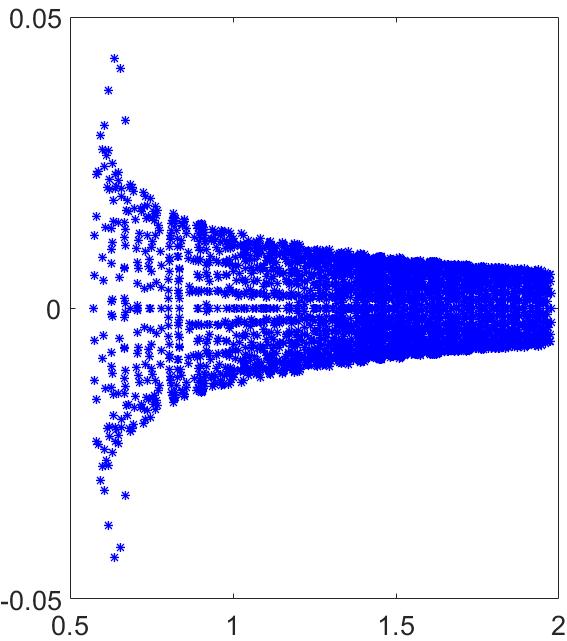}}\hspace{4mm}
	\subfigure[Eigenvalues of $P_{l}^{-1} \tilde{\mathcal{M}} P_{r}^{-1}$]
	{\includegraphics[width=2.5in,height=2.1in]{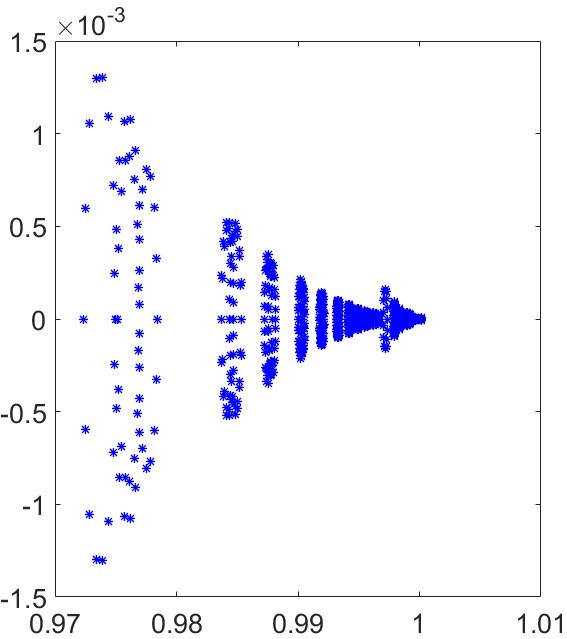}}
	\caption{Spectrum of $\tilde{\mathcal{M}}$, $P_{l}^{-1} \tilde{\mathcal{M}}$,
		$\tilde{\mathcal{M}} P_{r}^{-1}$ and $P_{l}^{-1} \tilde{\mathcal{M}} P_{r}^{-1}$
		for $(\alpha, \beta) = (0.9,1.9)$ and $M = N = 16$ for Example 3.}
	\label{fig3}
\end{figure}
\begin{figure}[ht]
	\centering
	\subfigure[Exact solution]
	{\includegraphics[width=2.7in,height=3.3in]{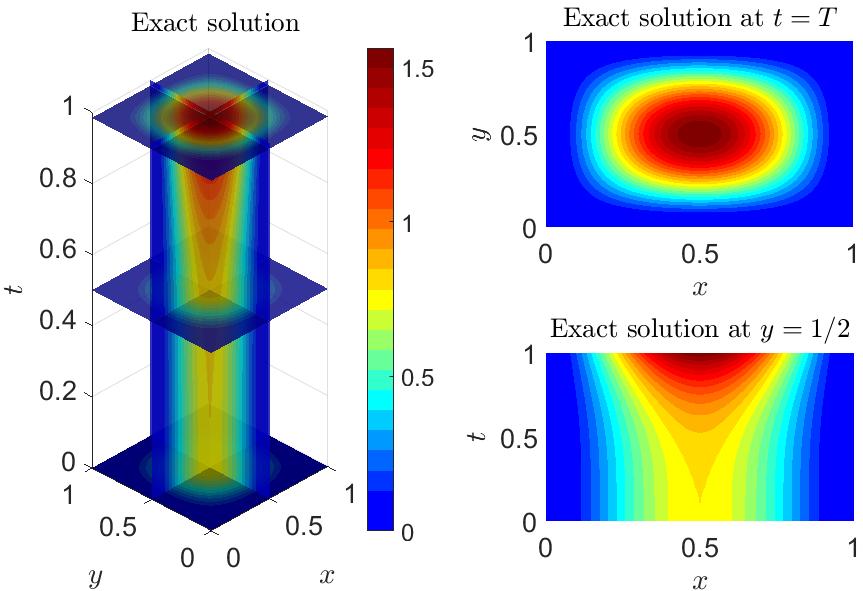}} \hspace{5mm}
	\subfigure[Numerical solution]
	{\includegraphics[width=2.7in,height=3.3in]{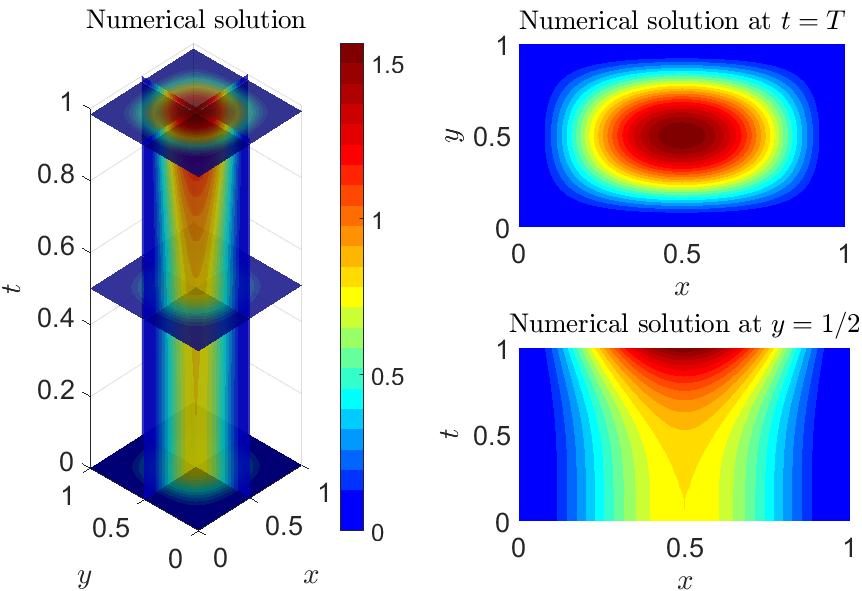}}
	\caption{Comparison of exact solution and numerical solution 
		for $(\alpha, \beta) = (0.35,1.5)$ and $M = N = 64$ for Example 3.}
	\label{fig4}
\end{figure}

\section{Concluding remarks}
\label{sec5}

In this article, we propose a bilateral preconditioning technique to solve 
the L2-type all-at-once system \eqref{eq2.5} arising from 
the TSFBTE \eqref{eq1.1}.
Firstly, combining the L2-type formula \cite{alikhanov2021high} and 
the fractional centered difference method \cite{ccelik2012crank,zhang2020exponential}, 
we propose and analyse an L2-type difference scheme \eqref{eq2.2} with $3 - \alpha$ accuracy in time 
to approximate Eq.~\eqref{eq1.1}.
Secondly, we derive the L2-type all-at-once system \eqref{eq2.5} based on this scheme.
In order to obtain the solution of Eq.~\eqref{eq2.5} efficiently, 
the left ($P_{l}$) and right ($P_{r}$) preconditioners are designed.
The condition number of the preconditioned matrix $P_{l}^{-1} \tilde{\mathcal{M}} P_{r}^{-1}$ is analyzed.
Finally, numerical examples are reported to show the performance of our method.
Moreover, in Example 3, we extend our bilateral preconditioning technique to solve 
the two-dimensional problem of Eq.~\eqref{eq1.1}.
It is worth mentioning that our method can be extended to solve linear all-at-once systems 
with graded time steps such as \cite{zhao2021preconditioning}.
In our future work, we will use the proposed preconditioning technique to solve semilinear problems,
e.g., the Volterra Allen-Cahn equation with weakly singular kernel \cite{gu2020parallel}.

\section*{Acknowledgments}
\addcontentsline{toc}{section}{Acknowledgments}
\label{sec6}

\textit{This research is supported by the National Natural Science Foundation
of China (Nos.~11801463 and 12101089), 
the Applied Basic Research Project of Sichuan Province (No.~2020YJ0007),
the Natural Science Foundation of Sichuan Province (Nos.~2022NSFSC1815 and 2023NSFSC1326)
and the Sichuan Science and Technology Program (No.~2022ZYD0006).}

\section*{References}
\addcontentsline{toc}{section}{References}
\bibliography{references}

\end{document}